%% file: main.tex
\documentclass[12pt]{article}
\usepackage{enumitem}
\usepackage[a4paper,margin=3cm,innermargin=3cm]{geometry}
\usepackage[latin9]{inputenc}
\usepackage{amsmath}
\usepackage{amssymb}
\usepackage{amstext}
\usepackage{amsfonts}
\usepackage{amsthm}
\usepackage{mathtools}
\usepackage{scrextend}
\usepackage{array}
\usepackage{multicol}
\usepackage{mathrsfs}
\usepackage{color}
\usepackage{hyperref}
\usepackage{bm}
\usepackage{fancyhdr}
\usepackage{svg}
\usepackage{float}
\usepackage[bottom]{footmisc}
\usepackage{csvsimple}
\usepackage{siunitx}
\usepackage{svg}
\usepackage{comment}

\usepackage{xurl}
\usepackage{optidef}

\usepackage{pgf,tikz,pgfplots}
\usetikzlibrary{decorations.markings}
\pgfplotsset{compat=1.15}
\usepackage{mathrsfs}
\usetikzlibrary{arrows}
\usepackage[numbers]{natbib}
\setcitestyle{number,open={[},close={]}}

\newcommand{\R}{\mathbb{R}}

\newcounter{number}
{\begin{list}%
{{\bf ~Proposition~\arabic{number}.}}{
\usecounter{number}
\setlength{\labelwidth}{0in}
\setlength{\leftmargin}{0in}
\setlength{\rightmargin}{0in}
\setlength{\itemsep}{1em}}}%
{\end{list}}

\newcommand{\solution}[1]{}

\newcommand{\reals}{{\mathbb R}}

\newcommand{\norm}[1]{ \left\|#1\right\|}

\newcommand{\fnorm}[1]{\norm{#1}_F}
\newcommand{\fnormbig}[1]{\big\| #1 \big\|_F}

\newcommand{\inner}[2]{\langle #1, #2 \rangle}

\usepackage[algo2e]{algorithm2e}
\usepackage{algorithmic}
\usepackage{algorithm, caption}

\DeclareMathOperator*{\argmax}{arg\,max}
\DeclareMathOperator*{\argmin}{arg\,min}

\DeclareMathOperator{\diag}{diag}

\DeclareMathOperator{\offdiag}{offdiag}

\DeclareMathOperator{\Span}{span}

\DeclareMathOperator{\vectorize}{vec}

\newcommand{\pesudoinverse}{\ssymbol{2}}

\def\@fnsymbol#1{\ensuremath{\ifcase#1\or *\or \dagger\or \ddagger\or
   \mathsection\or \mathparagraph\or \|\or **\or \dagger\dagger
   \or \ddagger\ddagger \else\@ctrerr\fi}}
\newcommand{\ssymbol}[1]{^{\@fnsymbol{#1}}}

\newcommand{\Normal}[2]{\mathcal{N}(#1, #2)}

  \newcommand{\Acal}{{\mathcal A}}
  
  \newcommand{\Ocal}{{\mathcal O}}

   \newcommand{\Prob}{\mathrm{Prob}}   %Probability

\newcommand{\suppress}[1]{}
\newtheorem{theorem}{Theorem}

\newtheorem{lemma}[theorem]{Lemma}

\newtheorem{remark}[theorem]{Remark}
\newtheorem{definition}[theorem]{Definition}
\newtheorem*{remark*}{Remark}

\newtheorem*{example*}{Example}

\numberwithin{equation}{section}

\DeclareMathOperator{\OB}{\mathcal{OB}}

\usepackage[many]{tcolorbox}

\title{A randomized algorithm for simultaneously diagonalizing symmetric matrices by congruence}

\author{Haoze He\footnotemark[1] \and Daniel Kressner\footnote{The work of both authors was supported by the SNSF research project \emph{Probabilistic methods for joint and
singular eigenvalue problems}, grant number: 200021L\_192049. \'Ecole Polytechnique F\'ed\'erale de Lausanne (EPFL), Institute of Mathematics, 1015 Lausanne, Switzerland. E-mails: \href{mailto:haoze.he@epfl.ch}{haoze.he@epfl.ch},  \href{mailto:daniel.kressner@epfl.ch}{daniel.kressner@epfl.ch}}}

\begin{document}
% \begin{itemize}
%     \item \textcolor{blue}{$\ker = \{0\}$ instead of $\ker = \emptyset$}
%     \item  \textcolor{blue}{Assumption $R > 0$ for Theorem 14 and Theorem 15.}
%     %\item  \textcolor{blue}{ Pham's loss $d$ and $n$.}
% \end{itemize}

\maketitle
\begin{abstract}
A family of symmetric matrices $A_1,\ldots, A_d$ is SDC (simultaneous diagonalization by congruence, also called non-orthogonal joint diagonalization) if there is an invertible matrix $X$ such that every $X^T A_k X$ is diagonal. In this work, a novel randomized SDC (RSDC) algorithm is proposed that reduces SDC to a generalized eigenvalue problem by considering two (random) linear combinations of the family. We establish exact recovery: RSDC achieves diagonalization with probability $1$ if the family is exactly SDC. Under a mild regularity assumption, robust recovery is also established: Given a family that is $\epsilon$-close to SDC then RSDC diagonalizes, with high probability, the family up to an error of norm $\Ocal(\epsilon)$. Under a positive definiteness assumption, which often holds in applications, stronger results are established, including a bound on the condition number of the transformation matrix. For practical use, we suggest to combine RSDC with an optimization algorithm. The performance of the resulting method is verified for synthetic data, image separation and EEG analysis tasks. It turns out that our newly developed method outperforms existing optimization-based methods in terms of efficiency while achieving a comparable level of accuracy.
\end{abstract}

\pagestyle{myheadings}
\thispagestyle{plain}

\input{introduction}

\input{preliminaries}

\input{exactrecovery}

\input{robustrecovery}

\input{implementation_details_RFFDIAG}

%\input{RFFDIAG_RRCG}

\input{numerical_experiments}

\input{conclusion}

% \clearpage
\bibliography{mybib.bib}
\bibliographystyle{abbrv}

\end{document}

%% file: introduction.tex
\section{Introduction}

A family of real symmetric $n\times n$ matrices $(A_1,\ldots, A_d)$ is called \emph{simultaneously diagonalizable by congruence} (SDC) if there is an invertible matrix $X$ such that $X^T A_k X$ is diagonal~\cite{Bustamante2020,jiang16,LeNguyen2022}. Under the stronger assumption that $X$ is orthogonal, such a family is usually called jointly diagonalizable (JD), which is equivalent to assuming that it is commuting. In the signal processing community, it is more common to refer to SDC as non-orthogonal joint diagonalization~\cite{Ablin19,absil06,Afsari08,Bouchard18, BSSsurvey}.
%SDC is more commonly referred to as .
%Following the convention used in recent theoretical developments on this topic~\cite{Bustamante2020,jiang16, LeNguyen2022, Sutton23}, we will use SDC as the default name.
In this work,  we extend our previous study~\cite{hekressner2024randomized} on randomized methods for JD to SDC.

SDC problems arise in a wide range of applications. Classically, in Blind Source Separation (BSS)~\cite{BSSsurvey} signal reconstruction is performed by applying SDC to covariance matrices associated with the observed noisy signals. It also plays a crucial role in robust Canonical Polyadic (CP) decomposition of tensors, see~\cite[Algorithm 3.1]{delathauwer06} and  Section~\ref{subsec:cpd} below. Recent applications of SDC include transfer learning \cite{Zhang22}, remote sensing~\cite{Khachatrian2021}, and quadratic programming \cite{jiang16,LeNguyen2022}.

Due to noise caused by, e.g., modeling, estimation or round-off errors, the SDC assumption is virtually never satisfied in practice. Instead, one encounters a family $(\tilde A_1, \ldots, \tilde A_d) = ( A_1+ \epsilon E_1, \ldots,  A_d+\epsilon E_d)$ that is \emph{nearly SDC}, that is, the members $A_k$ of an underlying (unknown) SDC family are perturbed by error matrices $\epsilon E_k$. % Unlike orthogonal JD of commuting matrices, it is not obvious to relate the problem of SDC to standard eigenvalue problems \cite{he2022randomized}.
Most existing algorithms for addressing such SDC problems proceed by considering an optimization problem of the form
\begin{equation}
\label{eq:optimization_problem}
    \min_{\tilde X \text{ invertible}} \mathcal{L}(\tilde X),
\end{equation}
where $\mathcal{L}(\tilde X)$ is a suitable loss function measuring off-diagonality for each transformed matrix $\tilde X^T \tilde A_k \tilde X$. A natural choice is
\begin{equation}
    \label{eq:off_diagonal_loss}
    \mathcal{L}(\tilde X) := \sum_{k=1}^{d} \big\|\offdiag(\tilde X^T \tilde A_k \tilde X) \big\|^2_F,
\end{equation}
where $\offdiag(\cdot)$ sets the diagonal entries of a matrix to zero and $\|\cdot\|_F$ denotes the Frobenius norm.
Various optimization methods, including a quasi-Newton method~\cite{Ziehe03} and an alternating Gauss-Newton iteration~\cite{uwedge} have been applied to this loss function. One obvious issue when working with~\eqref{eq:off_diagonal_loss} is that $\mathcal{L}(\tilde X)$ can be made arbitrarily small by simply rescaling $\tilde X$. A popular strategy to bypass this problem is to additionally impose the constraint that $\tilde X$ is in the oblique manifold $\OB(n,n)$, that is, each column of $\tilde X$ has Euclidean norm $1$. For example in~\cite{uwedge}, the columns of $\tilde X$ are renormalized after each iteration.
% that the columns of $\tilde X$ are normalized, that is, the additional constraint $\tilde X \in \OB(n,n)$ is imposed on~\eqref{eq:optimization_problem}, where  denotes \begin{equation}
% \label{eq:oblique_manifold}
%     \OB(n,n) = \{X \in \reals^{n \times n}: \diag(X^TX) = I_n\},
% \end{equation} need to be imposed.
Exploiting that $\OB(n,n)$ forms a Riemannian manifold, one can apply Riemannian optimization techniques~\cite{AbsMahSep2008,boumal2023intromanifolds} to SDC, such as Riemannian trust region \cite{absil06}, Riemannian BFGS \cite{Bouchard20} and Riemannian conjugate gradient~\cite{urdaneta2018RCGJD, bosner2023} methods. 

When $\tilde A_k$ is not only symmetric but also positive definite, a popular choice~\cite{Pham2001} of the loss function is given by
\begin{equation}
    \label{eq:pham_loss}
    \mathcal{L}(\tilde{X}) := \frac{1}{2n}\sum_{k = 1}^{d}[\log \det\diag(\tilde{X}^T\tilde{A}_k\tilde{X}) - \log\det(\tilde{X}^T\tilde{A}_k\tilde{X})],
\end{equation}
where $\diag(A) := A - \offdiag(A)$. As this loss function is invariant under column scaling~\cite{Pham2001}, no additional constraint needs to be imposed on $\tilde{X}$. In many signal processing applications~\cite{BSSsurvey}, the positive definite assumption is implied by the fact that $\tilde A_k$ is a sample covariance matrix. In this situation, the loss function~\eqref{eq:pham_loss} can be interpreted as the KL-divergence between zero-mean multivariate Gaussian distributions with covariance matrices $\tilde{X}^T\tilde{A}_k\tilde{X}$ and $\diag(\tilde{X}^T\tilde{A}_k\tilde{X})$~\cite{Bouchard18}. In the original paper~\cite{Pham2001}, the optimization problem \eqref{eq:optimization_problem} with loss function~\eqref{eq:pham_loss} is solved by a Jacobi-like method, successively applying invertible $2 \times 2$ transformations acting on pairs of columns. In \cite{Ablin19}, a quasi-Newton method is proposed that uses an approximate Hessian of~\eqref{eq:pham_loss}. In~\cite{Bouchard20}, a Riemannian optimization algorithm that optimizes~\eqref{eq:pham_loss} directly on the manifold of invertible matrices is developed and analyzed.

Optimization-based SDC algorithms enjoy two major advantages: Their convergence analysis inherits the theory of the underlying optimization algorithms~\cite{Ziehe03, uwedge} and they can be easily modified to suit more specific applications~\cite{uwedge,Pfister2019}. However, to the best of our knowledge, none of the existing methods is guaranteed to converge to the global minimum of~\eqref{eq:optimization_problem}. Moreover, as observed in~\cite{hekressner2024randomized}, optimization-based algorithms are often significantly slower than methods that utilize well-tuned linear algebra software.

In this paper, we propose and analyze a novel randomized SDC (RSDC) algorithm that is not only simpler but often also significantly faster than optimization-based SDC algorithms. % given a family $(A_1+E_1,\ldots,A_d+E_d)$ that is nearly SDC.
The core concept is simple: Similar to existing methods for JD~\cite{hekressner2024randomized,Ehler19}, for learning latent variable models~\cite{pmlr-v23-anandkumar12, anandkumar14, Anandkumar15}, for CP decomposition (see Section~\ref{sec:pencilbased}
 below) and for joint Schur decomposition~\cite{Corless1997}, we obtain $\tilde X$ by diagonalizing two random linear combinations of the matrices in the family. In contrast, existing work~\cite{Evert22,delathauwer06} on SDC in the context of the  CP decomposition utilizes two fixed matrices from the family. RSDC is straightforward to implement and exploits existing high-quality software for generalized eigenvalue problems implemented in, e.g., LAPACK~\cite{LAPACK,Kragstrom2006,Steel2023}. Moreover, unlike optimization-based methods, RSDC comes with recovery guarantees: Given a family that is (exactly) SDC, RSDC returns, with probability one, a matrix $X$ that transforms the family to diagonal form by congruence. Given a well-behaved (the precise meaning will be clarified in Section \ref{sec:preliminaries}) family that is nearly SDC, RSDC returns after column normalization, with high probability, an error~\eqref{eq:off_diagonal_loss} of $\Ocal(\epsilon)$. Furthermore, the accuracy can be further improved by refining the output of RSDC with the quasi-Newton method from~\cite{Ziehe03}.

\subsection{Connection with partially symmetric CP decomposition}\label{subsec:cpd}
% Given a third-order tensor $\mathsf{A} \in \reals^{n_1 \times n_2 \times n_3}$, CP decomposition factorizes $\mathsf{A}$ into a sum of rank-$1$ terms as follows  
% \[\mathsf{A} = \sum_{i=1}^{r} u^{(1)}_i \tensor u^{(2)}_i \tensor u^{(3)}_i\] 
% where $r$ is a positive integer, $u^{(1)}_i  \in \reals^{n_1}$, $u^{(2)}_i  \in \reals^{n_2}$, $u^{(3)}_i  \in \reals^{n_3}$ and $\otimes$ denotes the tensor product. The matrices $U^{(j)} = [ u_1^{(j)}\ \cdots \ u_r^{(j)} ]$ for $j=1,2,3$ are called \emph{factor matrices}.

Joint decompositions for matrix families are closely related to tensor decompositions. To explain this connection, let $\mathsf{A} \in \reals^{n\times n \times d}$ denote the tensor that has the  matrices of the family $(A_1,\ldots,A_d)$ as frontal slices, that is,
$\mathsf{A}(i,j,k) = A_k(i,j)$. Because each $A_k$ is symmetric, the 
tensor $\mathsf{A}$ is \textit{partially symmetric}, with symmetry in the first two modes \cite{carroll1970analysis,Kolda09}. Such a tensor is also called \textit{SFS} (symmetric frontal slices); see~\cite{DomanovDeLathauwer15,Evert23}.

For some $r\in \mathbb N$, the (general) CP decomposition aims at representing a tensor as a sum of $r$ rank-$1$ terms. The SFS-CP decomposition~\cite{DomanovDeLathauwer15} (also known as partially symmetric CP decomposition \cite[Section 5]{Kolda09} or INDSCAL~\cite{DomanovDeLathauwer15}) of an SFS-tensor $\mathsf{A}$ additionally requires that each rank-$1$ term inherits the partial symmetry:
\begin{equation} \label{eq:sfscp}
\mathsf{A} = \sum_{i=1}^{r} a_i \otimes a_i \otimes b_i, \quad a_i \in \reals^n, \quad b_i \in \reals^d,
\end{equation}
where $\otimes$ denotes the \emph{tensor product}. 

Suppose that the family is SDC, that is, $A_k = V D_k V^T$ for an invertible matrix $V \in \reals^{n \times n }$ and diagonal matrices $D_k \in \reals^{n \times n}$, $k = 1,\ldots,d$. Then the corresponding tensor $\mathsf{A}$ satisfies the SFS-CP decomposition
\[\mathsf{A} = \sum_{i=1}^{n} v_i \otimes v_i \otimes d_i\]
where  $v_i \in \reals^{n}$ denotes the $i$th column of $V$ and $d_i = [D_1(i,i), \ldots,D_d(i,i)]^T \in \reals^d$. Thus, SDC induces an SFS-CP decomposition~\eqref{eq:sfscp} with $r = n$. Vice versa, an SFS-CP decomposition~\eqref{eq:sfscp} of $\mathsf{A}$ with $r = n$ and invertible $[a_1,\ldots,a_n]$ implies that the frontal slices of $\mathsf{A}$ are SDC.

%We refer the readers to \cite{Kolda09} for a comprehensive introduction to CP decomposition for general tensors without the symmetry constraint. 

Both the CP decomposition and the SFS-CP decomposition~\eqref{eq:sfscp} with $r = n$ terms
are unique for generic $n\times n\times d$ tensors up to reordering and scaling, provided that $n,d\ge 2$; see~\cite{DomanovDeLathauwer13, DomanovDeLathauwer15}. 
% \textcolor{red}{??? add references ???}.
%The study of partially symmetric rank (also called SFS-rank) decompositions has been the subject of numerous recent works.  The natural problem  ``Is the SFS-rank equal to the CP rank?" -- also known as Comon's problem -- has been studied in \cite{GesmundoOnetoVentura19, Shitov24} with a recent counter example constructed in \cite{Shitov24}.
%Conditions on generic uniqueness of SFS-rank are established in \cite[Proposition 6.8]{DomanovDeLathauwer13} and subsequently refined in  \cite[Proposition 1.11, 1.12, 1.13]{DomanovDeLathauwer15}.Further, conditions that SFS rank equals to CP rank,
% \textcolor{red}{??? Move this to after Theorem 15. Make reference to this subsection in your updated discussion after Theorem 15! ???
% as well as the size of the neighborhood of an SFS rank-$r$ tensor to have a best SFS rank-$r$ approximation,  are recently provided in \cite[Theorem 2.1, 2.2]{Evert23} via positive definite linear combinations of frontal slices. Notably, the fact that positive-definiteness of the frontal slices enhances the stability of partial symmetric tensor decomposition in \cite{Evert23} aligns with our Theorem~\ref{thm:pd_prob_bound} where we show our proposed RSDC is more robust when the family is positive definite (to be defined in Section~\ref{sec:preliminaries}).}
On the algorithmic side, (approximate) SFS-CP decompositions can be computed using methods such as Gauss-Newton INDSCAL~\cite{GaussNewtonINDSCAL} or semi-definite relaxation~\cite{Nili22}.
These methods are not tailored to SDC, as they allow for $r > n$ terms and do not aim at ensuring the invertibility of $V$ when $r = n$.

\subsection{Pencil-based CP decomposition and its instability}
\label{sec:pencilbased}

As mentioned above, our proposed RSDC (Algorithm~\ref{alg:RSDC}) proceeds by taking two random linear combinations of the matrix family. Similar techniques have been proposed for computing the general CP decomposition of a tensor $\mathsf{A} \in \reals^{n \times n \times d}$. Such ``pencil-based" algorithms~\cite{Pencilbased} proceed by also taking two linear combinations of the frontal slices and then solving the generalized eigenvalue problem associated with this matrix pencil. For example, a suitably reordered generalized Schur decomposition of two fixed frontal slices is computed in \cite{Evert22}, whereas two input-dependent linear combinations of frontal slices are used in~\cite{DomanovDeLathauwer14, DomanovDeLathauwer17}. In~\cite{Evert22Subpace},  multiple pairs of linear combinations of frontal slices are used to enhance the robustness of the algorithm, including the use of random linear combinations.

In~\cite{Pencilbased}, the numerical stability of such pencil-based algorithms has been questioned by analyzing the condition number of a (unique) CP decomposition in terms of its  factor matrices~\cite{CondJoinMap}. Although the results in~\cite{Pencilbased} are stated for general CP decompositions, it is not unlikely that the arguments in~\cite{Pencilbased} carry over to SFS-CP decompositions and, hence, to SDC matrix families.
 
 First, Theorem 6.1 in~\cite{Pencilbased} shows that for any pair of linear combinations, there exists an adversarial input $\mathsf{A}\in \reals^{n\times n \times d}$ with $d \ge r+2$ such that the following holds: The condition number of the generalized eigenvalue problem (associated with those linear combinations of the frontal slices of $\mathsf A$) is arbitrarily larger than the condition number of the CP decomposition. Thus, any pencil-based algorithm that uses \emph{fixed} linear combinations can be defeated numerically, highlighting the importance of using random linear combinations.

 Second, Theorem 1.4 in~\cite{Pencilbased}
 considers random tensors in~$\reals^{n \times n \times d}$ that are constructed as a CP decomposition with $r$ rank-$1$ terms such that each term has fixed, arbitrary vectors in the first two modes and independent Gaussian random vectors in the third mode. It is shown that condition number for the CP decomposition of such tensors grows, with high probability, at least proportionally with $r^{2/(d-1)}$ as $r$ increases.
 As linear combinations of frontal slices do not change the distribution of the vectors in the third mode, this result indicates that using only two such linear combinations may have an unfavourable impact on the condition number. 

%The instability results in \cite{Pencilbased} also apply to the SDC case. However,   our proposed Algorithm~\ref{alg:RSDC} is not affected by \cite[Theorem 6.1]{Pencilbased} due to the use of random linear combinations, as demonstrated by our robust recovery results in Section~\ref{sec:robust_recovery}. 
It is worth stressing that the condition number studied in \cite{Pencilbased,CondJoinMap} considers the \emph{forward} error. In the context of SDC, this translates into analyzing the impact of small perturbations of the matrix family on the invertible matrix $X$ that effects the diagonalization by congruence. In this paper, we perform a \emph{backward error} analysis
for the approximate matrix $\tilde X$ returned by RSDC (Algorithm~\ref{alg:RSDC}), by quantifying the error of the approximate diagonalization effected by $\tilde X$. This error can be much smaller than the norm of $\tilde X - X$. Nevertheless, and possibly by coincidence, our main result (Theorem~\ref{thm:pd_prob_bound}) features an amplification factor $n^2$ that matches the 
CP decomposition condition number of the random tensor mentioned above for $r = n$ and $d = 2$.
% for using a pair of random linear combinations, we introduce an extra  factor in the off-diagonal (backward) error~\eqref{eq:off_diagonal_loss}, as will be shown in Section~\ref{subsec:prob_bounds}. This coincides with the condition number for computing CP decomposition of a random rank-$n$ tensor in $\reals^{n \times n \times 2}$, that is, a random generalized eigenvalue decomposition, with the randomness imposed in \cite[Theorem 1.4]{Pencilbased}, with high probability~\cite[Theorem 1.4]{Pencilbased}.
When $n$ is moderate and the level of noise remains low, our result guarantees that RSDC produces reasonably accurate results (with high probability). For large $n$ and/or a high noise level,
we will demonstrate \emph{numerically} how the potential numerical instability of RSDC can be mitigated by combining RSDC with optimization techniques.  Specifically, our numerical experiments in Section~\ref{sec:numerical_experiments} show that using the output from RSDC as a starting point of FFDIAG~\cite{Ziehe03} yields good accuracy and efficiency. It remains an open problem to theoretically analyze the quality of the output from RSDC as an initial guess for an optimization-based algorithm like FFDIAG.

\subsection{Organization}
The rest of this paper is organized as follows: In Section \ref{sec:preliminaries}, preliminaries about matrix pairs and families will be covered. In Section \ref{sec:exact_case}, the Randomized SDC (RSDC) algorithm is introduced and its exact recovery is established. In Section~\ref{sec:robust_recovery}, the robust recovery of RSDC is established under certain regularity and positive definiteness assumptions. Section~\ref{sec:implementation_details} covers important implementation details as well as the refinement of the output returned by RSDC. Section~\ref{sec:numerical_experiments} showcases the accuracy and efficiency of our algorithms through extensive numerical experiments, comparing with several state-of-the-art SDC solvers on both synthetic data and applications.

%% file: preliminaries.tex
\section{Preliminaries}\label{sec:preliminaries}

%  We always use the notation $\tilde{X}$ to indicate that it is a perturbed quantity of the quantity $X$. Throughout this paper, we will use Matlab notation for indexing matrices, $\diag(v)$ to denote a diagonal matrix with the vector $v$ as the diagonal entries. For a matrix $X$, we use $\Span(X)$ to denote its column span and $\ker(\cdot)$ to denote its kernel.

In this section, we discuss the basic properties of matrix families in the context of SDC.

For $d=2$, a matrix family becomes a matrix pair $(A,B) \in \reals^{n \times n} \times  \reals^{n \times n}$, closely associated with the generalized eigenvalue problem~\cite{matrixcomputation,StewartSun1990} for the matrix pencil $A -  \lambda B$. Such a pair is called \emph{regular} if the polynomial $\lambda \mapsto \det(A -  \lambda B)$ does not vanish. 
If a symmetric matrix pair $(A,B)$ is regular \emph{and} SDC then its Weierstrass canonical form~\cite[Chapter VI, Theorem 1.13]{StewartSun1990} is always diagonal. This implies that regularity is not sufficient to ensure the SDC property for a symmetric matrix pair $(A,B)$. For example, consider 
\begin{equation} \label{example:nonsdc}
A = \begin{bmatrix}
    0 & 1 \\
    1 & \epsilon
\end{bmatrix}, \quad B = \begin{bmatrix}
    0 & 1 \\
    1 & 0
\end{bmatrix}, \quad \epsilon \neq 0.\end{equation}
 If there was an invertible matrix $X$ such that $X^T A X = D_A$, $X^T B X = D_B$ are diagonal, $X^{-1}B^{-1}A X = D_B^{-1}D_A$ is diagonal. However,
 \[B^{-1}A = \begin{bmatrix} 1 & \epsilon \\ 0 & 1 \end{bmatrix}\] which is clearly not diagonalizable and leads to a contradiction. Thus $(A,B)$ is not SDC. The example~\eqref{example:nonsdc} also shows that arbitrarily small perturbations can destroy the SDC property, as the pair is clearly SDC for $\epsilon = 0$.

Let us now consider a family of $d \ge 2$ matrices: $A_1, \ldots, A_d \in \R^{n\times n}$. Given a vector of scalars 
$\mu \in \R^d$, we define
\[A(\mu):=\mu_1 A_1 + \ldots + \mu_d A_d.\]
In the following, we will now extend certain concepts, such as regularity, from matrix pairs to matrix families.

% \textcolor{red}{This definition is difficult to grasp. Is it equivalent to assuming that $A(\mu)$ is almost always invertible? If yes, please replace it. Btw, if 
% $A(\mu)$ is invertible for one $\mu$ it is invertible for almost all $\mu$.}
% \textcolor{blue}{Yes they are equivalent indeed! Because regular implies there is a linear combination $\mu'$ that $A(\mu')$ is invertible and therefore implies $A(\mu)$ is almost always invertible.}
\begin{definition}\label{def:regular_family}
    A family $(A_1,\ldots,A_d)$ with $A_k \in \reals^{n \times n}$ is called \emph{regular} if $\mu \mapsto \det( A(\mu) )$ does not vanish.
\end{definition}
Because $\det(A(\mu))$ is polynomial in the entries of $\mu$, it follows that a family is regular if and only if $A(\mu)$ is invertible for almost every $\mu \in \R^d$.

It clearly holds for all $\mu$ that
\begin{equation} \label{eq:inclkernels}
\ker(A_1)\cap \cdots \cap \ker(A_d) \subseteq \ker(A(\mu)).
\end{equation}
Thus, a necessary (but not sufficient) condition for regularity is that 
$\ker(A_1)\cap \cdots \cap \ker(A_d) = \{0\}$.
The following lemma identifies two situations in which equality holds in~\eqref{eq:inclkernels}. Here and in the following, $\reals_{>0}^{d}$ denotes the set of vectors of length $d$ with positive entries.
\begin{lemma} \label{lemma:intersection_kernels}
\begin{enumerate}
 \item[(i)] Let $A_1,\ldots,A_d$ be symmetric positive semidefinite. Then $\ker(A(\mu)) = \ker(A_1)\cap \cdots \cap \ker(A_d)$ holds for any $\mu \in \reals_{>0}^{d}$. 
 \item[(ii)] Let $(A_1,\ldots,A_d)$ be SDC. Then $\ker(A(\mu)) = \ker(A_1)\cap \cdots \cap \ker(A_d)$ holds
for almost every $\mu \in \reals^d$.
\end{enumerate}
\end{lemma}
\begin{proof}
(i) To establish the other inclusion in~\eqref{eq:inclkernels}, let $x \in \ker(A(\mu))$. Then 
    $x^TA(\mu)x  = \mu_1 x^TA_1x + \ldots + \mu_d x^TA_dx = 0.$
    Because each $A_k$ is positive semidefinite and $\mu$ is positive, this implies $x^TA_kx = 0$ and, in turn, $x \in \ker(A_k)$ for every $k$.
    
    (ii) Let $X$ be invertible such that $X^T A_k X = D_k$ is diagonal for $k = 1, \ldots, d$. This implies  $X^T A(\mu) X = D(\mu)$ and \[\ker(A_k) = X \ker(D_k), \quad \ker(A(\mu)) = X \ker(D(\mu)).\]
    For generic $\mu$, $D(\mu)$ has a zero diagonal entry if and only if all the corresponding diagonal entries of $D_k$ are zero as well. This shows $\ker(D(\mu)) = \ker(D_1)\cap \cdots \cap \ker(D_d)$.
\end{proof}

The following definition is partly motivated by Lemma~\ref{lemma:intersection_kernels}.
\begin{definition}\label{def:pd_family}
A family $(A_1,\ldots,A_d)$ of symmetric matrices is called \emph{positive definite} (PD) if
$A(\mu)$ is positive definite for every $\mu \in \reals_{>0}^{d}$.
\end{definition}
For $d=2$, Definition~\ref{def:pd_family} is stronger than the classical definition of a \emph{definite pair} \cite{StewartSun1990}, which only requires the existence of at least \emph{one} positive definite linear combination.
Using Lemma~\ref{lemma:intersection_kernels}, it follows that a family $(A_1,\ldots,A_d)$ is PD if and only if each $A_k$ is positive semidefinite and $\ker(A_1)\cap \cdots \cap \ker(A_d) = \{0\}$. Clearly, a PD family is also regular.

%% file: exactrecovery.tex
\section{Exact recovery of SDC}\label{sec:exact_case}

\subsection{Basic idea and template algorithm}

In~\cite{Afsari08}, it is explained why JD is generically a one-matrix problem, while SDC is generically a two-matrix problem. In analogy to our work on JD in~\cite[Algorithm 1]{hekressner2024randomized}, this suggests that one can hope to solve the SDC problem from two linear combinations. More formally, given a family $(A_1,\ldots,A_d)$ that is SDC, we form two random linear combinations
\[A(\mu) = \mu_1 A_1 + \ldots + \mu_d A_d,\quad  A(\theta) = \theta_1 A_1 + \ldots + \theta_d A_d,\]
and find an invertible $X$ (if it exists) such that $X^TA(\mu)X$ and $X^TA(\theta)X$ are diagonal. This idea leads to the Randomized SDC (RSDC) summarized in Algorithm~\ref{alg:RSDC}. Implementation details, especially concerning Line 3, will be covered in Section~\ref{sec:implementation_details}. We will establish exact and robust recovery of this algorithm in Sections~\ref{sub:exact_recover} and~\ref{sec:robust_recovery}, respectively. 
\begin{algorithm}[H]
    \caption{\textbf{R}andomized \textbf{S}imultaneous \textbf{D}iagonalization via \textbf{C}ongruence (RSDC)}
    \textbf{Input:} \text{An SDC family $(A_1,\ldots,A_d)$.}\\
     \textbf{Output:} \text{Invertible matrix $X$ such that $X^T A_k X$ is diagonal for $k = 1,\ldots,d$.} \\[-0.5cm]
    \begin{algorithmic}[1]
    \label{alg:RSDC}
        \STATE Draw $\mu, \theta \in \R^d$ from specified distributions.
        \STATE Compute $A(\mu) = \mu_1 A_1 + \cdots + \mu_d A_d$, $A(\theta) = \theta_1 A_1 + \cdots + \theta_d A_d$.
        \STATE Compute an invertible $X$ such that $X^TA(\mu)X$ and $X^TA(\theta)X$ are diagonal.
    \RETURN $ X $
    \end{algorithmic}
    \end{algorithm}

\subsection{Exact recovery}
\label{sub:exact_recover}

Trivially, any pair of linear combinations $(A(\mu), A(\theta))$ of an SDC family is also SDC. It turns out that a congruence transformation diagonalizing this pair almost always diagonalizes the whole family. To show this, we exploit an existing connection between SDC and JD or, equivalently, commutativity.
\begin{lemma}[{\cite[Theorem 3.1 (ii)]{LeNguyen2022}}]
\label{lemma:iff_condition_sdc}
   A family $(A_1,\ldots A_d)$ of symmetric matrices is SDC if and only if there is an invertible matrix $P$ such that $(P^T A_1 P,\ldots, P^T A_d P)$ is a commuting family.
\end{lemma}
\begin{remark}
\label{rmk:transformation}
    If there is a positive definite linear combination $A(\theta)$ and, hence, $A(\theta) = LL^T$ admits a Cholesky factorization, one can choose $P = L^{-1}$ in Lemma \ref{lemma:iff_condition_sdc}; see~\cite[Theorem 2.1]{LeNguyen2022}.% In BSS, $A(\theta)$ can be chosen as the covariance matrix of the observerd signals \cite{Belouchrani1997}.
\end{remark}

Next, we extend our exact recovery result~\cite[Theorem 2.2]{hekressner2024randomized} on joint diagonalization from orthogonal to invertible similarity transformations.

\begin{lemma}[Joint diagonalization by similarity]
\label{lemma:jd_by_similarity}
Let $(A_1,\ldots,A_d)$ be a commuting family such that $A_k$ is diagonalizable for every $k=1,\ldots,d$. For almost every $\mu \in \R^d$, the following statement holds: If $X$ is an invertible matrix such that  $X^{-1}A(\mu)X$ is diagonal, then $X^{-1}A_kX$ is also diagonal for $k = 1,\ldots,d$.
%For almost every $\mu \in \reals^d$, if $X \in \reals^{n \times n}$ is any invertible matrix such that $X^{-1}%A(\mu)X$ is diagonal where $A(\mu) = \mu_1 A_1 + \cdots + \mu_d A_d$, then 
%$X^{-1}A_kX$ is diagonal for each $k 
%\in \{1,\ldots,d\}$. 
\end{lemma}
\begin{proof}
The proof of this result is along the lines of the proof of 
 \cite[Theorem 2.2]{hekressner2024randomized}; we mainly include it for the sake of completeness. Without loss of generality, we may assume that the first $n_1 \geq 1$ columns of $X$ span the eigenspace $\mathcal{X}_1$ belonging to an eigenvalue $\lambda_1(\mu)$ of $A(\mu)$. Then 
 \[X^{-1}A(\mu)X = \begin{bmatrix}
    \lambda_1(\mu)I_{n_1} & 0 \\
    0 & A_{22}(\mu)
 \end{bmatrix}, \quad X^{-1}A_kX = \begin{bmatrix}
     A_{11}^{(k)} & A_{12}^{(k)}\\
     A_{21}^{(k)} & A_{22}^{(k)}
 \end{bmatrix},\]
 with $A_{11}^{(k)} \in \R^{n_1\times n_1}$.
 As similarity transformations preserve commutativity, the matrix $X^{-1}A(\mu)X$ commutes with each $X^{-1}A_kX$. This implies
 \[A_{12}^{(k)}(\lambda_1(\mu)I - A_{22}(\mu)) = 0, \quad (A_{22}(\mu) - \lambda_1(\mu)I)A_{21}^{(k)} = 0.\]
 Because $\lambda_1(\mu)$ is not an eigenvalue of $A_{22}(\mu)$, we conclude that \[X^{-1}A_kX = \begin{bmatrix}
     A_{11}^{(k)} & 0\\
     0 & A_{22}^{(k)}
 \end{bmatrix}.\] 
 Because $A_{11}^{(1)}, \ldots, A_{11}^{(d)}$ are commuting and diagonalizable and $\mu_1 A_{11}^{(1)} + \cdots + \mu_d A_{11}^{(d)} = \lambda_1(\mu)I_{n_1}$ , it follows that each $A_{11}^{(k)}$ is diagonal, in fact, a multiple of the identity matrix~\cite[Lemma 2.1]{hekressner2024randomized} for almost every $\mu \in \reals^d$. Because the family $A_{22}^{(1)}, \ldots, A_{22}^{(d)}$ satisfies the assumptions of the lemma, we can conclude the proof by induction.
% Since $X^{-1}A_iX$ commutes with $X^{-1}A_jX$ for all $i \neq j$, the family $\{A^{(k)}_{11}\}, \{A^{(k)}_{22}\}$ are two commuting families. Since $X^{-1}A_kX$ is diagonalizable, $A_{11}^{(k)}, A_{22}^{(k)}$ are also diagonalizable. By , for almost every $\mu$, each $A_{11}^{(k)}$ has $n_1$ equal eigenvalues and is, therefore, . The proof is completed by induction.
\end{proof}

We are now ready to state and prove our exact recovery result for Algorithm \ref{alg:RSDC}.
\begin{theorem}[Exact recovery]
\label{thm:exact_recovery}
    Let $(A_1,\ldots,A_d)$ be SDC. Then the following holds for almost every $(\mu,\theta) \in \reals^d \times \reals^d$: If $X$ 
is an invertible matrix such that $X^T A(\mu) X$ and $X^T A(\theta) X$ are diagonal, then the matrices $X^TA_kX$ are diagonal for $k = 1,\ldots,d$.
\end{theorem}
\begin{proof}
By Lemma \ref{lemma:iff_condition_sdc}, there exists an invertible matrix $P \in \reals^{n \times n}$ such that $(B_1,\ldots, B_d)$ with $B_k:= P^TA_kP$ is a commuting family. Note that $\ker(A_k) = P\ker(B_k)$ and therefore $\ker(A(\mu)) = P\ker(B(\mu))$. This relation together with Lemma \ref{lemma:intersection_kernels} imply that 
\begin{equation}\label{eq:kernel_B_mu}
    \ker(B(\mu)) = \ker(B_1) \cap \cdots \cap \ker(B_d)
\end{equation} holds for almost every $\mu \in \reals^d$. For the rest of the proof, we assume that this relation holds.

%Invoking the exact recovery of RJD on the commuting family $\mathcal{B}$, we get the exact recovery if $\mu,\theta$ follow a continuous distribution.
By the assumptions, $D_{\mu}:= X^TA(\mu)X$ and $D_{\theta} := X^TA(\theta)X$ are diagonal and it thus follows that 
\begin{equation}
\label{eq:Y_diagonalizes_two}
    Y^T B (\mu) Y = D_{\mu}, \quad Y^T B(\theta) Y = D_{\theta},
\end{equation}
where $Y := P^{-1}X$ is invertible. It remains to show that $Y^TB_kY$ is diagonal for $k = 1,\ldots,d$.

%\textcolor{red}{I had difficulties following the rest of the proof for the following reasons:
%\begin{itemize}
% \item When saying that $B(\mu)$ is invertible, it sounds like you worked with a fixed $\mu$ (and not that almost every $\mu$ leads to invertibility)? However, %then suddenly in the second part you argue with genericity. Something does not fit together here; please reflect other your argument and rewrite.
%\end{itemize}
%}
First consider $ \ker(B_1)\cap \cdots \cap \ker(B_d) = \{0\}$. By \eqref{eq:kernel_B_mu}, $B(\mu)$ is invertible and therefore $D_{\mu}$ is invertible, and
\[Y^{-1}B(\mu)^{-1}B(\theta)Y = D_{\mu}^{-1}D_{\theta}.\]
Note that the family $\big( B(\mu)^{-1}B_1,\ldots, B(\mu)^{-1}B_d \big)$ is also commuting because $B(\mu)^{-1}$ commutes with each $B_k$. Moreover, each matrix $B(\mu)^{-1}B_k$ is symmetric and thus diagonalizable. Note that $Y^{-1}B(\mu)^{-1}B(\theta)Y$ is diagonal. This allows us to apply Lemma~\ref{lemma:jd_by_similarity} to this family and conclude that 
\begin{equation}
    \label{eq:jd_B_k}
     C_k:= Y^{-1} B(\mu)^{-1}B_k Y
\end{equation}
is diagonal for almost every $\theta$.
By \eqref{eq:Y_diagonalizes_two},
$Y^{-1} = D_{\mu}^{-1}Y^TB(\mu)$ and plugging this relation into~\eqref{eq:jd_B_k} gives 
\[Y^TB(\mu)B(\mu)^{-1}B_kY = D_{\mu}C_k,\]
which concludes the proof.

We now treat the case when $ \ker(B_1)\cap \cdots \cap \ker(B_d) \neq \{0\}$ by deflation. Then $B(\mu)$ is not invertible and therefore $D_\mu$ is not invertible. We assume, without loss of generality, that its zero diagonal entries appear first. Thus, 
$\ker(B(\mu)) = Y\ker(D_\mu )$ with suitable partition $Y = \begin{bmatrix}
    Y_1 & Y_2
\end{bmatrix}$ where 
$\Span(Y_1) = \ker(B(\mu))$.
Considering the QR factorization of Y
\[Y = \begin{bmatrix}
    Q_1 & Q_2
\end{bmatrix} \begin{bmatrix}
    R_1 & *\\
    0 & R_2
\end{bmatrix},\]
we have $\Span(Q_1) = \ker(B(\mu)) $.
Then
\begin{align*}
Y^TB(\mu)Y &=  \begin{bmatrix}
    R^T_1 & 0\\
    * & R^T_2
\end{bmatrix}
\begin{bmatrix}
    Q^T_1 \\ Q^T_2
\end{bmatrix}B(\mu) \begin{bmatrix}
    Q_1 & Q_2
\end{bmatrix} \begin{bmatrix}
    R_1 & *\\
    0 & R_2
\end{bmatrix} \\
& = \begin{bmatrix}
    R^T_1 & 0\\
    * & R^T_2
\end{bmatrix}
\begin{bmatrix}
    0 & 0 \\
    0 & Q_2^TB(\mu)Q_2
\end{bmatrix}\begin{bmatrix}
    R_1 & *\\
    0 & R_2
\end{bmatrix}\\
&= \begin{bmatrix}
    0 & 0\\
    0 &  R_2^T Q_2^TB(\mu)Q_2 R_2^T
\end{bmatrix}.
\end{align*}

From~\eqref{eq:kernel_B_mu}, it follows that $\ker(B_k) \subset \ker(B(\mu))$.  Thus, defining $B'_k := Q_2^T B_k Q_2$,
\[Y^TB_kY = \begin{bmatrix}
    0 & 0 \\ 
    0 & R_2^TB'_kR_2
\end{bmatrix}\]
with $\ker( B'_1) \cap \cdots \cap \ker( B'_d) = \{ 0 \}$.
Because $Y$ is invertible, $R_2$ is invertible. Moreover, the family $\{B'_1,\ldots,B'_d\}$ is commuting, $R_2^{T} B'(\mu)R_2, R_2^{T} B'(\theta)R_2$ are diagonal and $\ker(B'_1) \cap \cdots \cap \ker(B'_d) =   \{ 0 \}$. This reduces the problem to the case covered in the first part of the proof.
\end{proof}
Results similar to the one by Theorem~\ref{thm:exact_recovery} have been derived in the context of tensor decompositions, for the $n\times n\times d$ tensor with frontal slices $A_1,\ldots, A_d$. In particular, Theorem 2.1 in~\cite{smoothedAnalysis} establishes essentially the same result but under an additional assumption phrased in terms of the Kruskal rank. In terms of the notation introduced in Section~\ref{sub:robust_theorems} below, this assumption amounts to imposing $m = n$. %Translated to our setting, this assumption requires that no 
%,  does not require the transformation matrix $X$ to be essentially unique~\cite{Afsari08, TensorDforSP} and does not impose any gap assumption on the middle matrices $D_k = X^TA_kX$.
\begin{remark}
    \label{rmk:invertible_theta}
    For later purposes, we note that the first part of the proof also allows us to conclude that the statement of Theorem \ref{thm:exact_recovery} holds for almost every $\theta \in \reals^d$ if $\mu \in \reals^d$ is chosen such that $A(\mu)$ is invertible. 
\end{remark}

%% file: robustrecovery.tex
\section{Robust recovery}\label{sec:robust_recovery}

We now consider the situation when Algorithm~\ref{alg:RSDC} is applied to a family $(\tilde A_1, \ldots, \tilde A_d)$ that is not SDC itself but close to an SDC family $(A_1,\ldots,A_d)$. Algorithm~\ref{alg:RSDC} then proceeds by forming a pair $\big( \tilde A(\theta), \tilde A(\mu) \big)$ of two random linear combinations  of $\tilde A_1, \ldots, \tilde A_d$. We aim at showing a robustness result for Algorithm~\ref{alg:RSDC} of the following form: Any invertible $\tilde X$ such that $\tilde X^T \tilde A(\mu) \tilde X$ and $\tilde X^T \tilde A(\theta) \tilde X$ are diagonal nearly diagonalizes the whole family $(\tilde A_1, \ldots, \tilde A_d)$. 
While the exact recovery result of Theorem~\ref{thm:exact_recovery} applies to general SDC families, our robustness analysis will only consider regular~(Definition \ref{def:regular_family}) and, in particular, positive definite~(Definition \ref{def:pd_family}) SDC families. This restriction is due to the fact that non-regular families are not well-behaved under perturbations.

\subsection{Perturbation results} \label{sec:perturbation}

This section collects preliminary results on the perturbation theory for generalized eigenvalue problems~\cite{StewartSun1990,kressner05}.

Given a regular matrix pair $(A,B)$, a scalar $\lambda$ is called a (finite) eigenvalue with associated eigenvector $x \not=0$ if
$Ax = \lambda B x$. The \emph{eigenspace} associated with $\lambda$ is $\ker(A - \lambda B)$ and $\lambda$ is a semi-simple eigenvalue if the dimension of its eigenspace matches its algebraic multiplicity.
Deflating subspaces generalize eigenspaces of matrix pairs; just as invariant subspaces generalize eigenspaces of matrices. Concretely, a pair of subspaces $(\mathcal{X},\mathcal{Y})$ is a pair of right/left \emph{deflating subspaces} of $(A,B)$ if $\dim(\mathcal{X}) = \dim(\mathcal{Y})$ and $A\mathcal{X}, B\mathcal{X}$ are both contained in $\mathcal{Y}$.
 Note that the subspace $\mathcal{Y}$ is uniquely defined by $\mathcal{X}$~\cite{kressner05} and that any eigenspace of $(A,B)$ is also a right deflating subspace. 

Let us consider orthogonal matrices 
$\begin{bmatrix}
    X_1 & X_\perp
\end{bmatrix} \in \reals^{n \times n}$, $\begin{bmatrix}
    Y_1 & Y_\perp
\end{bmatrix} \in \reals^{n \times n}$ partitioned such that $X_1,Y_1 \in \reals^{n\times \ell}$. Then, by definition, $\big(\Span(X_1),\Span(Y_1)\big)$
is a pair of deflating subspaces if and only if
\begin{align}\label{eq:genschur}
\begin{split}
\begin{bmatrix}
    Y_1 & Y_\perp
\end{bmatrix}^TA \begin{bmatrix}
    X_1 & X_\perp
\end{bmatrix} & = \begin{bmatrix}
    A_{11} & A_{12} \\ 0 & A_{22}
\end{bmatrix}, \quad A_{11} \in \reals^{\ell \times \ell}, \\
\begin{bmatrix}
    Y_1 & Y_\perp
\end{bmatrix}^TB \begin{bmatrix}
    X_1 & X_\perp
\end{bmatrix} &= \begin{bmatrix}
    B_{11} & B_{12} \\ 0 & B_{22}
\end{bmatrix}, \quad B_{11} \in \reals^{\ell \times \ell}.
\end{split}
\end{align}

\begin{lemma}
\label{lemma:gevp_deflating_subspace_perturabtion}
With the notation introduced above, suppose that $\Span(X_1)$ is the eigenspace associated with a semi-simple, finite eigenvalue 
$\lambda \in \R$ of $(A,B)$. Given $E ,F \in \reals^{n\times n}$, the perturbed matrix pair 
$(A+\epsilon E,B+\epsilon F)$ has a right deflating subspace $\mathcal{\tilde X}_1 = \Span(\tilde X_1)$
 such that
\begin{equation} \label{eq:tildex1}
    \tilde X_1 = X_1 + \epsilon X_\perp Z_1 + \Ocal(\epsilon^2) \in \reals^{n \times n}
\end{equation}
holds for sufficiently small $\epsilon$, with the matrix 
$Z_1 \in \R^{(n-\ell)\times \ell}$ uniquely 
given by
\begin{equation} \label{eq:Z1}
     Z_1  = \big(Y_\perp^T( \lambda B - A)X_\perp \big)^{-1} Y_\perp^T( E -  \lambda F)X_1.
\end{equation}
\end{lemma}
\begin{proof}
Existing perturbation expansions, see~\cite[Corollary 4.1.7]{jgsun2002} and~\cite[Theorem 2.8]{kressner05}, state that~\eqref{eq:tildex1} holds with
    \begin{equation}\label{eq:Z_1_perturbation_analysis}
        \vectorize(Z_1) = C_{11}\vectorize(E_{21}) + C_{12}\vectorize(F_{21}),
    \end{equation}
    where 
        $E_{21}:= Y_\perp^TEX_1$, $F_{21}:= Y_\perp^TFX_1$, and
    \begin{align*}
        C_{11}  &= (B^T_{11} \otimes I_{n-\ell})( A^T_{11} \otimes B_{22 }-B^T_{11} \otimes A_{22})^{-1},\\  C_{12}  &= (-A^T_{11} \otimes I_{n-\ell})( A^T_{11} \otimes B_{22 }-B^T_{11} \otimes A_{22})^{-1},
    \end{align*}
    with $\vectorize(\cdot)$ denoting the vectorization of a matrix. Because $\mathcal{X}_1$ is an eigenspace associated with $\lambda$, we have that $A_{11} B_{11}^{-1} = \lambda I_\ell$ and, thus, the relations above simplify to
    \begin{equation*}C_{11}  = I_\ell \otimes (\lambda B_{22} -  A_{22})^{-1},\quad  C_{12}   =  - \lambda I_\ell \otimes (\lambda B_{22} - A_{22})^{-1}.\end{equation*}
    Inserted into~\eqref{eq:Z_1_perturbation_analysis}, this shows~\eqref{eq:Z1} by using basic properties of the Kronecker product.
    The uniqueness of $Z_1$ follows from the uniqueness of the Taylor expansion together with the fact that $X_\perp$ has full column rank.
\end{proof}

For the purpose of analyzing the impact of perturbations on diagonalization by congruence, we require the following variant of Lemma \ref{lemma:gevp_deflating_subspace_perturabtion}.

\begin{lemma}
\label{lemma:gevp_perturabtion}
Let $(A,B)$ be SDC and let $X = \begin{bmatrix}
    X_1 & X_2
\end{bmatrix} \in \R^{n\times n}$ be an invertible matrix such that: $X^TAX$, $X^TBX$ are diagonal and $X_1 \in \R^{n\times \ell}$ is an orthonormal basis of the eigenspace associated with a semi-simple eigenvalue $\lambda \in \R$. 
Given $E,F \in \reals^{n \times n}$, the perturbed matrix pair $(A+\epsilon E, B + \epsilon F)$ has a right deflating subspace $\tilde{\mathcal{X}}_1 = \Span(\tilde X_1)$ such that 
\begin{equation} \label{eq:modpertexpansion}
    \tilde  X_1 = \hat X_1 + \epsilon X_2Z_1 + \mathcal O(\epsilon^2) \in \reals^{n \times n}
\end{equation}
holds for sufficiently small $\epsilon$,
where
\begin{equation}\label{eq:X_1_hat}
\hat X_1 =  X_1(I- \epsilon X_1^TX_2(X_2^T(\lambda B - A)X_2)^{-1}X_2^T(E - \lambda F)X_1)
\end{equation}
and $Z_1 \in \reals^{(n-\ell) \times \ell}$ is uniquely given by 
\begin{equation}
    Z_1 = (X_2^T(\lambda B - A)X_2)^{-1}X_2^T( E - \lambda F)X_1.
\end{equation}
\end{lemma}
\begin{proof}
As $X_1$ spans a right deflating subspace we can find orthogonal matrices $\begin{bmatrix}
        X_1,X_\perp
    \end{bmatrix}$, $\begin{bmatrix}
        Y_1,Y_\perp
    \end{bmatrix}$ satisfying~\eqref{eq:genschur}.
       Because $X_2^T A X_1 = 0$, $X_2^T B X_1 = 0$, and $X$ is invertible, we have
    \[\Span(X_\perp) = \Span((I-X_1X_1^T)X_2), \quad \Span(Y_\perp) = \Span(X_2).\]
    Therefore, there exist invertible matrices $C_X, C_Y$ such that \[
    %\label{eq:normalization}
    X_\perp = (I-X_1X_1^T)X_2C_X, \quad Y_\perp = X_2C_Y.
    \]
   Inserted into the perturbation expansion of 
    Lemma~\ref{lemma:gevp_deflating_subspace_perturabtion}, this yields 
    \begin{align*}
        \tilde  X_1 &= X_1 + \epsilon X_\perp \big(Y_\perp^T( \lambda B - A)X_\perp \big)^{-1} Y_\perp^T( E -  \lambda F)X_1 + \Ocal(\epsilon^2)  \\
        &= X_1 +  \epsilon(I-X_1X_1^T)X_2(X_2^T(\lambda B - A) (I-X_1X_1^T)X_2)^{-1}X_2^T(E - \lambda F)X_1 + \Ocal(\epsilon^2) \\
        &=X_1 +  \epsilon(I-X_1X_1^T)X_2(X_2^T(\lambda B - A)X_2)^{-1}X_2^T(E - \lambda F)X_1 + \Ocal(\epsilon^2) \\
        &= X_1(I- \epsilon X_1^TX_2(X_2^T(\lambda B - A)X_2)^{-1}X_2^T(E - \lambda F)X_1)  \\
        &+ \epsilon X_2(X_2^T(\lambda B - A)X_2)^{-1}X_2^T(E - \lambda F)X_1 + \Ocal(\epsilon^2),
    \end{align*}
    where we used $(\lambda B - A) (I-X_1X_1^T)  = (\lambda B - A)$ in the second equality. It can be directly verified that the last expression matches~\eqref{eq:modpertexpansion}.
\end{proof}

\subsection{A structural bound}
\label{sub:robust_theorems}

Let $\Acal = (A_1,\ldots,A_d)$ be a regular SDC family and let $X$ be an invertible matrix such that $D_k := X^T A_k X$ is diagonal.
Let $\theta \in \R^d$ be an arbitrary fixed vector such that $A(\theta)$ and, hence, $D(\theta) = \theta_1 D_1 + \cdots + \theta_d D_d$ is invertible. Then the diagonal matrix $\Lambda_k = D_k D(\theta)^{-1}$ contains the eigenvalues of the pair $(A_k, A(\theta))$. Collecting all diagonal entries at position $i$ into a vector $[\Lambda_1(i,i), \ldots, \Lambda_d(i,i)] \in \R^d$, we may reorder the columns of $X$ so that identical vectors are grouped together. In other words, we may assume that
\begin{equation} \label{eq:clusterlambdak}
  \Lambda_k = \begin{bmatrix}
              \lambda^{(k)}_1 I_{n_1} \\ & \ddots & \\ && 
              \lambda^{(k)}_m I_{n_m}
             \end{bmatrix}, \
             [\lambda^{(1)}_i,\ldots,\lambda^{(d)}_i] \not= [\lambda^{(1)}_j,\ldots,\lambda^{(d)}_j],\ \forall i\not=j.
\end{equation}
It is simple to see that this grouping does not depend on the particular choice of $\theta$.
%Note that \[[\Lambda_1(i,i), \ldots, \Lambda_d(i,i)] = [D_1(i,i), \ldots, D_d(i,i)] / \inner{\theta}{[D_1(i,i), \ldots, D_d(i,i)] }\]
%where $\inner{\cdot}{\cdot}$ denotes the inner product. Therefore, the grouping defined in \eqref{eq:clusterlambdak} only depends on the directions of $[D_1(i,i), \ldots, %D_d(i,i)], i=1,\ldots,n$ and thus does not depend on the particular choice of $\theta$.

By~\eqref{eq:clusterlambdak}, the matrix $\Lambda(\mu) = D(\mu) D(\theta)^{-1} = \mu_1 \Lambda_1 + \cdots + \mu_d \Lambda_d$ takes the following form for almost every $\mu$:
\begin{equation} \label{eq:clusterlambdamu}
\Lambda(\mu) = 
\begin{bmatrix}
              \lambda_1(\mu) I_{n_1} \\ & \ddots & \\ && 
              \lambda_m(\mu) I_{n_m}
             \end{bmatrix},\ \lambda_i(\mu) \not= \lambda_j(\mu),\ \forall i\not= j.
\end{equation}
Partitioning $X = [X_1,\ldots,X_m]$ with $X_i \in \R^{n\times n_i}$, we thus have that $\Span(X_i)$ is the eigenspace of $A(\mu) A(\theta)^{-1}$ associated with the eigenvalue $\lambda_i(\mu)$. Without loss of generality, we may assume that the columns of $X_i$ are orthonormal.

 We are now ready to state the following structural bound on robust recovery.
\begin{theorem}[Structural bound]
\label{thm:structural_bound}
Given a regular SDC family $\mathcal A = (A_1,\ldots, A_d)$ with $d\ge 2$, consider the perturbed family
$\widetilde{\mathcal A} = (\tilde A_1,\ldots, \tilde A_d)$ with $\tilde A_k = A_k + \epsilon E_k$ such that $\fnorm{E_1}^2 + \cdots + \fnorm{E_d}^2 = 1$ and $\epsilon > 0$. For $\theta \in \reals^d$ such that $A(\theta)$ is invertible, define the diagonal matrices $\Lambda_k$ and $\Lambda(\mu)$ as in~\eqref{eq:clusterlambdak}--\eqref{eq:clusterlambdamu}. Then the following is true for almost every $\mu \in \reals^d$: For any invertible matrix $\tilde X \in \reals^{n \times n}$ such that $\tilde X^T \tilde A(\mu) \tilde X$ and $\tilde X ^T \tilde A(\theta) \tilde X$ are diagonal, it holds that 
\begin{align*}
&\frac{1}{\|\tilde X\|_2^2} \Big(\sum_{k=1}^{d}\big\|\offdiag(\tilde X^T \tilde A_k \tilde X) \big\|_F^2 \Big)^{1/2}\\\
\leq& \sqrt{d} \big( D_{\max} \|E(\mu)\|_F + ( \lambda^\mu_{\max} D_{\max} + \lambda_{\max})  \|E(\theta)\|_F\big)\epsilon + \epsilon + \mathcal O(\epsilon^2)
\end{align*}
where $\lambda^\mu_{\max} := \|\Lambda(\mu)\|_2$, $\lambda_{\max} := \max_k \|\Lambda_k\|_2$, and
\begin{equation}\label{eq:lambda_max}
    D_{\max}  := 
\max_{i,k}\{\|(\Lambda_k - \lambda^{(k)}_{i}I)(\lambda_i(\mu)I - \Lambda(\mu))\pesudoinverse\|_2\}.
\end{equation}

\end{theorem}
%Note that the quantity $D_{\max}$ will be bounded by Lemma \ref{lemma:failure_prob}.
\begin{proof} 
% Without loss of generality, we may assume that $X' = \begin{bmatrix}
%     X'_1 & \ldots & X'_m
% \end{bmatrix}$ such that each block $X'_i$ has orthonormal columns and spans an eigenspace of $A(\theta)^{-1}A(\mu)$. We have $A(\mu)X' = A(\theta)X'\Lambda$ and thus $X'^TA(\mu)X' = X'^TA(\theta)X'\Lambda$. Since $X'^TA(\mu)X'$ is symmetric, $X'^TA(\theta)X'$ and $\Lambda$ commute. As a result,
% \begin{align*}
%     X'^TA(\mu)X' &= \diag(\begin{bmatrix}
%     {X'_1}^{T}A(\mu)X'_1 & \ldots &  {X'_m}^{T}A(\mu)X'_m
% \end{bmatrix}),\\ X'^TA(\theta)X' &= \diag(\begin{bmatrix}
%     {X'_1}^{T}A(\theta)X'_1 & \ldots &  {X'_m}^{T}A(\theta)X'_m
% \end{bmatrix})
% \end{align*}
% where ${X'_i}^{T}A(\mu)X'_i$ commute with ${X'_i}^{T}A(\theta)X'_i$ for $i = 1,\ldots,m$. Therefore there exist orthogonal $Q_i$ such that $Q_i^T{X'_i}^{T}A(\mu)X'_iQ_i$ and $Q_i^T{X'_i}^{T}A(\theta)X'_iQ_i$ are both diagonal for $i = 1,\ldots,m$. 
%Define $X = \begin{bmatrix}
%    X'_iQ_i & \ldots & X'_mQ_m
%\end{bmatrix}$.  is a matrix of common generalized eigenvectors (Definition \ref{def:commen_generalized_eigenvector}) for the SDC %family $\Acal$ with respect to $A(\theta)$. 

Let $X = [X_1,\ldots,X_d]$ be the matrix described above. In particular, $X_i$ is, for almost every $\mu$, an orthonormal basis for the eigenspace of $(A(\mu),A(\theta))$ associated with the semi-simple eigenvalue $\lambda_i(\mu)$. Because the norm of the off-diagonal part is invariant under reordering the columns of $\tilde X$, we may assume without loss of generality that $\tilde X = [\tilde X_1, \ldots, \tilde X_m]$, where each $\tilde X_i \in \R^{n\times n_i}$ spans a right deflating subspace for the perturbed matrix pair $(\tilde A(\mu),\tilde A(\theta))$ that is 
close to $\Span(X_i)$ in the sense of Lemmas~\ref{lemma:gevp_deflating_subspace_perturabtion} and~\ref{lemma:gevp_perturabtion}.

% Now consider the perturbed matrix pair $\big(A(\mu)+E(\mu),A(\theta)+E(\theta)\big)$, it is clear that $\tilde X$ is a matrix of eigenvectors for $\big(A(\mu)+E(\mu),A(\theta)+E(\theta)\big)$.  Since the perturbation is sufficiently small, without loss of generality, by  perturbation of the generalized eigenvalues~\cite{StewartSun1990}, after suitable permutation, we can assume that $\tilde X = \begin{bmatrix}
%     \tilde X_1, \ldots, \tilde X_m
% \end{bmatrix}$
% where columns of $\tilde X_i$ form a basis (not necessarily orthonormal) of the perturbed right deflating subspace $\tilde{\mathcal{X}}_i$ close to the original eigenspace $\mathcal{X}_i$.

Using that $\tilde X^T \tilde A(\theta) \tilde X \Lambda_k$ is diagonal and that a diagonal modification does not affect the off-diagonal norm, we obtain that
\begin{align}
        &\offdiag(\tilde X^T A_k \tilde X) \nonumber\\
        =& \offdiag(\tilde X^T A_k \tilde X - \tilde X^T \tilde A(\theta) \tilde X \Lambda_k)\nonumber\\ 
        =& \offdiag(\tilde X^T A_k \tilde X - \tilde X^T  A(\theta) \tilde X \Lambda_k) +  \epsilon\cdot \offdiag(\tilde X^T E(\theta)  \tilde X\Lambda_k).\label{eq:offdiag_decomp}
    \end{align}
We rewrite the first term in~\eqref{eq:offdiag_decomp} as follows:
\begin{equation}
 \tilde X^T ( A_k \tilde X -   A(\theta) \tilde X \Lambda_k) 
    = \tilde X^T \big[
         (A_k - \lambda^{(k)}_{1}A(\theta))\tilde X_1,\ \ldots,\  (A_k - \lambda^{(k)}_{m}A(\theta))\tilde X_m\big].
         \label{eq:off_main_decomp}
\end{equation}
Considering the first block $(A_k - \lambda^{(k)}_{1}A(\theta))\tilde X_1$ in the right-hand side of~\eqref{eq:off_main_decomp}, we set $X_{-1} = \begin{bmatrix}
    X_2 & \ldots & X_m
\end{bmatrix}$ and apply Lemma \ref{lemma:gevp_perturabtion} to obtain, for sufficiently small $\epsilon$, a basis $\bar X_1$ such that $\Span(\hat X_1) = \Span(\tilde X_1)$ and 
\begin{equation}
    \label{eq:pertubation_basis}
    \bar X_1 = \hat X_1 + \epsilon X_{-1}(X^T_{-1}(\lambda_1(\mu)A(\theta) - A(\mu))X_{-1})^{-1}X_{-1}^T F_1 X_1  + \mathcal O(\epsilon^2)
\end{equation}
where  $F_1: = E(\mu) - \lambda_1(\mu)E(\theta)$ and $\hat X_1$ is a basis of $\Span(X_1)$ satisfying $\hat X_1 = X_1 + \mathcal O(\epsilon)$. As there is an invertible matrix $C_1 \in \R^{n_1\times n_1}$ such that $\tilde X_1 = \bar X_1 C_1$, we can rewrite~\eqref{eq:pertubation_basis} as
\begin{equation}
    \label{eq:pertubation_eigenvector}
    \tilde X_1 = \hat X_1 C_1+ \epsilon X_{-1}(X^T_{-1}(\lambda_1(\mu)A(\theta) - A(\mu))X_{-1})^{-1}X_{-1}^T F_1 X_1 C_1  + \mathcal O(\epsilon^2).
\end{equation}
 Using that $(A_k - \lambda^{(k)}_{1}A(\theta)) \hat X_1 = 0$, this implies
\begin{align*}
    &(A_k - \lambda^{(k)}_{1}A(\theta))\tilde X_1\\ 
    =& \epsilon (A_k - \lambda^{(k)}_{1}A(\theta))X_{-1}(X_{-1}^T(\lambda_1(\mu)A(\theta) - A(\mu))X_{-1})^{-1}X_{-1}^TF_1X_1C_1 + \mathcal O(\epsilon^2)\\
    =&\epsilon X^{-T}X^T(A_k - \lambda^{(k)}_{1}A(\theta))X(X^T(\lambda_1(\mu)A(\theta)-A(\mu))X)\pesudoinverse X^TF_1X_1C_1 + \mathcal O(\epsilon^2)\\
    =& \epsilon X^{-T} (\Lambda_k -\lambda^{(k)}_{1}I)(\lambda_1(\mu)I - \Lambda(\mu))\pesudoinverse X^{T}F_1X_1C_1 + \mathcal O(\epsilon^2).
\end{align*}
%The last equality above is obtained by the following fact:
%\begin{align*}
%    &X^T(A_k - \lambda^{(k)}_{1,\theta}A(\theta))X(X^T(\lambda_1(\mu)A(\theta)-A(\mu))X)\pesudoinverse \\
%    =&X^T\big(A_k -\lambda^{(k)}_{1,\theta}A(\theta)\big)X \big(X^TA(\theta)X\big)^{-1}\big(X^TA(\theta)X\big)(X^T(\lambda_1(\mu)A(\theta)-A(\mu))X)\pesudoinverse.
%\end{align*}
For general $i \in \{1,\ldots,m\}$, we obtain in an analogous fashion that
\begin{equation}
\label{eq:perturbation_bound}
    (A_k - \lambda^{(k)}_{i}A(\theta))\tilde X_i
    = \epsilon X^{-T} \tilde D_i^{(k)} X^{T}F_iX_iC_i + \mathcal O(\epsilon^2),
\end{equation}
where $F_i = E(\mu) - \lambda_{i}(\mu)E(\theta)$,
$\tilde D_i^{(k)}
= (\Lambda_k - \lambda^{(k)}_{i}I)(\lambda_{i}(\mu)I - \Lambda(\mu))\pesudoinverse$ and $C_i \in \R^{n_i\times n_i}$ is invertible. Note that
%    we obtain that $D_i^{(k)}$ commutes with $C$ for $i = 1, \ldots,m, k = 1, \ldots,d$ for almost every $\mu$.
$D_{\max} = \max_{i,k}\{\|\tilde D_i^{(k)}\|_2\}$.
Plugging \eqref{eq:perturbation_bound} into \eqref{eq:off_main_decomp} yields
\begin{align}
    &\tilde X^T A_k \tilde X - \tilde X^T  A(\theta) \tilde X \Lambda_k \nonumber\\
    &= \epsilon \tilde X^T X^{-T} \begin{bmatrix}
         \tilde D_1^{(k)}X^{T}F_1X_1C_1 & \cdots &   \tilde D_m^{(k)} X^{T}F_mX_mC_m
    \end{bmatrix} + \mathcal O(\epsilon^2) \nonumber \\
    &=\epsilon C^T\begin{bmatrix}
         \tilde D_1^{(k)}X^{T}F_1X_1 & \cdots &  \tilde D_m^{(k)} X^{T}F_mX_m
    \end{bmatrix}C + \mathcal O(\epsilon^2) \nonumber \\
    &= \epsilon \begin{bmatrix}
         \tilde D_1^{(k)}(XC)^{T}F_1X_1 & \cdots &  \tilde D_m^{(k)} (XC)^{T}F_mX_m
    \end{bmatrix}C + \mathcal O(\epsilon^2) \nonumber,
    %\label{eq:off_error}
\end{align}
where $C= \diag(C_1,\ldots, C_m)$. The second equality above exploits the property $\tilde X = \bar X C = XC + \Ocal(\epsilon)$, and the third equality uses that the block diagonal matrix $C$ commutes with the diagonal matrix $\tilde D_i^{(k)}$ for each $i$. Thus, we get
\begin{align}
    &\|\offdiag(\tilde X^T A_k \tilde X - \tilde X^T  A(\theta) \tilde X \Lambda_k ) \|_F \nonumber\\
    \leq & \epsilon  D_{\max} \| XC \|_2 \big\|\begin{bmatrix}
         F_1X_1 & \cdots & F_mX_m
    \end{bmatrix}C \big\|_F + \mathcal O(\epsilon^2)\nonumber\\
    \leq& \epsilon D_{\max} \|XC\|_2 (\|E(\mu)\|_F \|XC\|_2 + \|E(\theta)\|_F \|X\Lambda(\mu)C\|_2) + \mathcal O(\epsilon^2)\nonumber\\
    =& \epsilon D_{\max} \|XC\|_2 (\|E(\mu)\|_F \|XC\|_2 + \|E(\theta)\|_F \|XC\Lambda(\mu)\|_2) + \mathcal O(\epsilon^2)\nonumber\\
    \leq& \epsilon \|\tilde X\|_2^2D_{\max} (\|E(\mu)\|_F + \lambda^\mu_{\max} \|E(\theta)\|_F) + \mathcal O(\epsilon^2), \label{eq:firsttermbound}
\end{align}
where the third equality uses that $C$ and $\Lambda(\mu)$ commute and the last inequality again uses $\tilde X = XC + \Ocal(\epsilon)$.

It remains to bound the second term in~\eqref{eq:offdiag_decomp}:  \[\big\|\offdiag(\tilde X^T E(\theta)  \tilde X \Lambda_k)\big\|_F \leq   \|\tilde X\|_2^2 \lambda_{\max} \|E(\theta)\|_F .\]
Plugging this bound together with~\eqref{eq:firsttermbound} into~\eqref{eq:offdiag_decomp} and using the triangle inequality establish the result of the theorem.
% Therefore, the overall off-diagonal error structural bound will be
% \begin{align*}
% &(\sum_{k=1}^{d}\big\|\offdiag(\tilde X^T \tilde A_k \tilde X)\big\|_F^2)^{-1/2} / \|\tilde X\|_2^2\nonumber\\
% \leq&  \sqrt{d}D_{\max}  (\|E(\mu)\|_F + \lambda_{\max} \|E(\theta)\|_F) + (\sum_{k=1}^{d}\|\Lambda_k\|^2_2)^{1/2}\|E(\theta)\|_F + \epsilon + \mathcal O(\epsilon^2).
% \end{align*}
%
\end{proof}

\subsection{Probabilistic bounds for robust recovery}
\label{subsec:prob_bounds}

By analyzing the quantities involved in the bound of Theorem~\ref{thm:structural_bound}, we derive probabilistic bounds for random $\theta,\mu$, specifically for Gaussian random vectors ($ \sim \Normal{0}{I_d}$). Recall that $m$ is the number of distinct eigenvalue vectors defined in \eqref{eq:clusterlambdak}. %The following lemmas address the most criticial quantity $D_{\max}$ as well as $\lambda_{\max}^\mu$ and $\lambda_{\max}$. 
\begin{lemma}
\label{lemma:failure_prob}
Let $D_{\max}$ be defined as in~\eqref{eq:lambda_max} and assume that $\mu \sim \Normal{0}{I_d}$.
Then  the inequality
$
 D_{\max}\|\mu\|_2 \leq R_0
$
holds for any $R_0>0$ with probability at least $1-\sqrt{\frac{d}{2\pi}}\frac{m(m-1)}{R_0}$.
\end{lemma}
\begin{proof}
    The proof is along the lines of the proof of Theorem 3.4 in~\cite{hekressner2024randomized}. We include the proof for the sake of completeness. Following the notation introduced in \eqref{eq:clusterlambdak},
\[ D_{\max}  =
\max_{i,k}\{\|(\Lambda_k - \lambda^{(k)}_{i}I)(\lambda_i(\mu)I - \Lambda(\mu))\pesudoinverse\|_2\}
 = \max_{i,j,k, i\neq j} \big\{\big| \vec{\Delta}^{(ij)}_k / \inner{\vec{\Delta}^{(ij)}}{\mu} \big|\big\},
\]
with the nonzero vector $\vec{\Delta}^{(ij)} := [\lambda_j^{(1)}  - \lambda_i^{(1)}, \ldots, \lambda_j^{(d)} - \lambda_i^{(d)}] \in \reals^{d}$.  
For a fixed pair $i,j$ such that $i\neq j$, we have that
\begin{align*}
    &\Prob\big( \max_k \big\{\big| \| \mu \|_2\vec{\Delta}^{(ij)}_k / \inner{\vec{\Delta}^{(ij)}}{\mu} \big| \big\} \geq R_0 \big)\\
    = & \Prob\big( \big| \inner{\vec{\Delta}^{(ij)}}{\mu / \| \mu \|_2} / \vec{\Delta}^{(ij)}_{k^*} \big| \leq 1 / R_0 \big)\\
    \leq & \Prob\big( \big| \inner{\vec{\Delta}^{(ij)} / \|\vec{\Delta}^{(ij)}\|_2}{\mu / \| \mu \|_2} \big| \leq 1 / R_0 \big) \leq  \sqrt{\frac{2d}{\pi}}\frac{1}{R_0}
\end{align*}
where in the first equality $k^*:= \argmax_k\{| \vec{\Delta}^{(ij)}_k|\}$ and the last inequality is a standard result in the literature~\cite{Dixon1983}; see also~\cite[Lemma 3.1]{hekressner2024randomized}. Applying the union bound for the $m(m-1)/2$ different pairs of $i,j$ concludes the proof.
\end{proof}

\begin{lemma}
\label{lemma:control_theta}
With the eigenvalue components $\lambda^{(k)}_i$ defined as in~\eqref{eq:clusterlambdak}, suppose that
\begin{equation} 
\big\|\big[\lambda^{(1)}_i,\ldots, \lambda^{(d)}_i \big] \big\|_2 \|\theta\|_2 \leq C, \quad i = 1,\ldots, m,  \label{eq:condlambda}
\end{equation}  
holds for some $C> 0$. 
Then the constants $\lambda^\mu_{\max},\lambda_{\max}$ of Theorem~\ref{thm:structural_bound} satisfy
\begin{equation} \label{inequlambda}
 \lambda^\mu_{\max}\|E(\theta)\|_F \leq C \|\mu\|_2, \quad \lambda_{\max} \|E(\theta)\|_F \leq C.
\end{equation}
If $\theta \sim \Normal{0}{I_d}$ then the inequality~\eqref{eq:condlambda}, and thus~\eqref{inequlambda}, holds with probability at least $1-\sqrt{\frac{2d}{\pi}}\frac{m}{C}$.
\end{lemma}
\begin{proof}
Using $\|E_1\|_F^2 + \cdots + \|E_d\|_F^2 = 1$ and the Cauchy-Schwarz inequality, we obtain that $\|E(\theta)\|_F \le \|\theta\|_2$. Because of \[ \lambda_{\max}^\mu = \max_i \sum_k |\mu_k| |\lambda_i^{(k)}| \le \max_i \big\|\big[\lambda^{(1)}_i,\ldots, \lambda^{(d)}_i \big] \big\|_2 \|\mu\|_2,\]
we obtain the first inequality in~\eqref{inequlambda} from~\eqref{eq:condlambda}. Similarly, $\lambda_{\max} \|E(\theta)\|_F \le \max_{i,k} |\lambda_i^{(k)}| \|\theta\|_2$ implies the second inequality.

It remains to bound the probability that~\eqref{eq:condlambda} fails. For this purpose, we use that
$\lambda^{(k)}_i$ is a diagonal element of $D_k D(\theta)^{-1}$ and, hence, there is $j$ such that
$\lambda^{(k)}_i = D_k(j,j) / \langle \Xi_j, \theta \rangle$ with $\Xi_j = \begin{bmatrix}
    D_1(i,i) & \ldots & D_d(i,i)
\end{bmatrix}$. In turn, we have that
\[
 \Prob\big( \big\|\big[\lambda^{(1)}_i,\ldots, \lambda^{(d)}_i \big] \big\|_2 \|\theta\|_2 \geq C \big) = \Prob\big( \|\Xi_j\| \|\theta\|_2 / |\langle \Xi_j, \theta \rangle|  \geq C \big) \le \sqrt{\frac{2d}{\pi}} \frac{1}{C},
\]
where the last inequality is, again, a standard result in the literature~\cite{Dixon1983}.
%; see also~\cite[Lemma 3]{he2022randomized}.
Applying the union bound for $i = 1,\ldots,m$ concludes the proof.
\end{proof}

\begin{theorem}[Main theorem for regular families]
\label{thm:npd_prob_bound}
Under the setting of Theorem~\ref{thm:structural_bound}, assume that $\theta, \mu$ are independent Gaussian random vectors. Let $\tilde X$ 
be any invertible matrix such that $\tilde X^T \tilde A(\mu) \tilde X$ and $\tilde X^T \tilde A(\theta) \tilde X$ are diagonal. Then, for any $R > 0$, it holds that
    \begin{align*}
\Prob\Bigg(\frac{1}{\|\tilde X\|_2^2}\Big(\sum_{k=1}^{d}\fnormbig{\offdiag(\tilde{X}^T\tilde{A}_k\tilde{X})}^2\Big)^{1/2} &\leq (R^2 + 1) \epsilon  + \Ocal(\epsilon^2)\Bigg) \\ &\geq 1 - \sqrt{\frac{10}{\pi}} \frac{m^{3/2} d^{3/4}}{R}.
\end{align*}
\end{theorem}
\begin{proof}
Given $R>0$, we set $R_0 = R \sqrt{m} / \sqrt{5 \sqrt{d}} > 0$. By the union bound, the bounds of Lemma~\ref{lemma:failure_prob} for $R_0$ and  Lemma~\ref{lemma:control_theta} for $C = 2 R_0 / m$ hold, with probability at least
\begin{equation} \label{eq:success}
 1-\sqrt{\frac{d}{2\pi}}\frac{m(m-1)}{R_0} - \sqrt{\frac{2d}{\pi}}\frac{m}{C} \ge 
 1-\sqrt{\frac{2d}{\pi}}\frac{m^2}{R_0}  = 1 - \sqrt{\frac{10}{\pi}} \frac{m^{3/2} d^{3/4}}{R},
\end{equation}
Because $d \ge 2$, this bound becomes negative when $R_0 < m^2$ and, in turn, the result of the theorem trivially holds. We may therefore assume $R_0 \ge m^2$ in the following.
Since the invertibility of $A(\theta)$ holds almost surely, one can apply
Theorem~\ref{thm:structural_bound}, which yields the following bound:
\textcolor{black}{\begin{align*}
&\frac{1}{\|\tilde X\|_2^2} \Big(\sum_{k=1}^{d}\big\|\offdiag(\tilde X^T \tilde A_k \tilde X) \big\|_F^2 \Big)^{1/2} \\
\leq& \sqrt{d} \big( D_{\max} \|E(\mu)\|_F + ( \lambda^\mu_{\max} D_{\max} + \lambda_{\max})  \|E(\theta)\|_F\big)\epsilon + \epsilon + \mathcal O(\epsilon^2) \\
\le& \sqrt{d} ( R_0 + C R_0 + C)\epsilon + \epsilon + \mathcal O(\epsilon^2) \\
=& \sqrt{d} \big( m / R_0 + 2  + 2 / R_0\big)R_0^2/m\cdot \epsilon + \epsilon + \mathcal O(\epsilon^2) \\
\le & 5\sqrt{d}   R^2_0 /m \cdot \epsilon + \epsilon + \mathcal O(\epsilon^2)  \\
 = & ( R^2 + 1 ) \epsilon + \mathcal O(\epsilon^2).
\end{align*}}%
Here, the second inequality follows from the bounds of Lemmas~\ref{lemma:failure_prob} and~\ref{lemma:control_theta}. The third inequality
uses that $R_0 \ge m^2 \geq 1$.
\end{proof}

The bound of Theorem~\ref{thm:npd_prob_bound} implies that the output error is $\Ocal(m^3 d^{3/2} \delta^{-2} \epsilon)$, with failure probability at most $\delta$. In practice, the empirical output error is observed to be $\Ocal(\delta^{-1})$ for regular families, which remains open to be proved. However, in the positive definite case, this bound improves to $\Ocal(m^2d^2\delta^{-1}\epsilon)$ when making a fixed choice for $\theta$. Note that, $m$, as introduced in \eqref{eq:clusterlambdak}, is bounded by the matrix size $n$.
\begin{theorem}[Main theorem for PD families]
\label{thm:pd_prob_bound}
Under the setting of Theorem~\ref{thm:structural_bound}, assume that $\mu$ is a Gaussian random vector and $\theta = \big[1/d,\ldots,1/d \big]$. Further assume that $\Acal$ is a positive definite family. If $\tilde X$ 
is any invertible matrix such that $\tilde X^T \tilde A(\mu) \tilde X$ and $\tilde X^T \tilde A(\theta) \tilde X$ are diagonal then, for any $R > 0$, we have
\begin{align*}
\Prob\Big(\frac{1}{\|\tilde X\|_2^2}\Big(\sum_{k=1}^{d}\fnormbig{\offdiag(\tilde{X}^T\tilde{A}_k\tilde{X})}^2\Big)^{1/2} &\leq (1+R)\epsilon  + \Ocal(\epsilon^2)\Big) \\ &\geq 1 - \frac{3}{\sqrt{2\pi}} \frac{d^2m(m-1)}{R}.
\end{align*}
\end{theorem}
\begin{proof}
When $\theta = \big[1/d,\ldots,1/d \big],$ with the notation introduced at the beginning of Section \ref{sub:robust_theorems}, $\Lambda_k =D_k D(\theta)^{-1}$ and thus  $\lambda_i^{(k)} \leq d$. 
Then the bound of Lemma~\ref{lemma:control_theta} always holds with $C = d$.
  Moreover, for $R_0 > 0$, the bound of Lemma~\ref{lemma:failure_prob} holds with probability at least
\[1 - \sqrt{\frac{d}{2\pi}}\frac{m(m-1)}{R_0}.\] 
Noting that the invertibility of $A(\theta)$ is satisfied because $\Acal$ is a PD family, we can insert the bounds of the lemmas into the result of Theorem~\ref{thm:structural_bound} and obtain that the off-diagonal error scaled by $\|\tilde X\|_2^{-2}$ is, up to $\mathcal O(\epsilon^2)$, bounded by 
\begin{eqnarray*}
    && \sqrt{d} \big( D_{\max} \|E(\mu)\|_F + ( \lambda^\mu_{\max} D_{\max} + \lambda_{\max})  \|E(\theta)\|_F\big)\epsilon + \epsilon \\ 
    &\le& \sqrt{d} (R_0 + d R_0 + d)\epsilon + \epsilon \leq 3d^{3/2}R_0 \epsilon + \epsilon.
    \end{eqnarray*}
    Performing the substitution $R_0 = R / (3 d^{3/2})$ completes the proof.
\end{proof}

%given an SFS rank-$r$ tensor $\mathsf A$,  Theorem 2.1 in~\cite{{Evert23}} shows that the existence of one positive definite linear combination of frontal slices of $\mathsf A$ ensures that the SFS rank equals the CP rank. Furthermore, 
Theorem \ref{thm:pd_prob_bound} provides an indication that positive-definiteness can improve the reliability of SDC. A similar insight appears in the context of closely related SFS-CP decomposition, introduced in Section~\ref{subsec:cpd}. Specially, given an SFS rank-$r$ tensor $\mathsf A$, Theorem 2.2 in~\cite{{Evert23}} shows that the optimal positive-definite linear combination of frontal slices of $\mathsf A$ controls the size of the neighborhood around $\mathsf A$ where every tensor admits a best SFS rank-$r$ approximation. It is worth noting that Theorem~\ref{thm:pd_prob_bound} requires the family to be positive-definite (Definition~\ref{def:pd_family}), whereas results in~\cite{{Evert23}} rely on the existence of just one positive-definite linear combination of frontal slices.

For SDC problems arising in BSS, it is common to utilize only one matrix from a PD family (instead of forming an average like in Theorem \ref{thm:pd_prob_bound}) during the so-called prewhitening~\cite{Belouchrani1997}.
In other words, the pair $(\tilde A(\mu), \tilde A_k)$, with a fixed matrix $\tilde A_k$, is simultaneously diagonalized by congruence for determining the SDC transformation for the whole family. However, the result and proof of Theorem \ref{thm:pd_prob_bound} suggest that considering only one matrix, no matter how well-chosen, instead of the average $\tilde A(\theta)$ could potentially lead to a large error, as indeed sometimes observed in the BSS community~\cite{cardoso1994performance, uwedge, Bouchard20}.

\subsection{Controlling the condition number for a PD family}

While the results from Theorems~\ref{thm:npd_prob_bound} and~\ref{thm:pd_prob_bound} ensure that the transformation matrix  $\tilde X$ is invertible and the bounds do not depend on the norm of $\tilde X$, they do not guarantee a well-conditioned $\tilde X$. In the PD case, an asymptotic bound on the condition number $\kappa(\tilde X) = \|\tilde X\|_2 \|\tilde X^{-1}\|_2$ can be established for a specific variant of Algorithm~\ref{alg:RSDC}.

Consider the symmetric positive definite SDC family $\Acal = \{A_k\}$ and the perturbed family $\tilde \Acal = \{A_k + \epsilon E_k\}$. Instead of choosing $\theta$ randomly in the two linear combinations $\tilde A(\theta)$ and $\tilde A(\mu)$, we fix $\theta=[1/d,\ldots,1/d]$ as done in Theorem~\ref{thm:pd_prob_bound}.  By positive definiteness,  $A(\theta)$ is positive definite. While the family $\tilde \Acal$ might not be positive definite, $\tilde A(\theta)$ is still positive definite for sufficiently small $\epsilon$, which allows us to compute the Cholesky factorization $\tilde A(\theta) = \tilde L \tilde L^T$, with $\tilde L$ lower triangular.
Then, by the spectral decomposition, we determine an orthogonal matrix $\tilde Q$ such that $\tilde Q^T\tilde L^{-1}\tilde A(\mu) \tilde L^{-T}\tilde Q$ is diagonal. As a result, the matrix
\begin{equation} \label{eq:X_cholesky}
    \tilde X =\tilde  L^{-T} \tilde  Q
\end{equation} simultaneously diagonalizes $\tilde A(\mu)$ and $\tilde A(\theta)$ by congruence. We will use the described procedure to realize Line 3 of Algorithm~\ref{alg:RSDC} for determining $\tilde X$.
The following lemma shows that $\tilde X$ is well conditioned, for sufficiently small $\epsilon$, if there exists at least one well-conditioned convex combination in the family $\Acal$. For simplicity, we also assume that all matrices $A_k$ have spectral norm $1$, which can always be achieved by scaling.
\begin{lemma}\label{lemma:control_condition}
With the notation and assumptions introduced above, suppose that $\|A_1\|_2 = \cdots = \|A_d\|_2 = 1$, the matrix $\tilde X$ defined in~\eqref{eq:X_cholesky} satisfies, for $\epsilon$ sufficiently small,
    \[\kappa(\tilde X) \leq d \sqrt{\kappa^*} + \Ocal(\epsilon).\]
     where 
    $\kappa^*:= \min_{\mu \in \Delta^d}\{\kappa(A(\mu))\}$ and $\Delta^d := \{\mu\in\reals^d,\mu_1 + \cdots + \mu_d =1, \mu_k \geq 0\}$.
\end{lemma}
\begin{proof}
By definition,
   \begin{equation}\label{eq:cond_opt}
       \kappa^* = \min_{\mu \in \Delta^d} \frac{\max_{v \in \mathcal{S}^{n-1}}\inner{\mu}{\Xi_v} }{\min_{v \in \mathcal{S}^{n-1}}\inner{\mu}{\Xi_v} }
   \end{equation} 
   where $\Xi_v:= [v^TA_1v,\ldots,v^TA_dv] \in \reals^d$ and $\mathcal{S}^{n-1} = \{ v \in \R^n\colon \|v\|_2 = 1\}$. For any $\mu \in \Delta^d$, the denominator in~\eqref{eq:cond_opt} is bounded from above by
   \begin{equation*}
      \min_{v \in \mathcal{S}^{n-1}}\{\inner{\mu}{\Xi_v}\}  \leq \min_{v \in \mathcal{S}^{n-1}}\{\|\mu\|_1\|\Xi_v\|_\infty\}  \leq d\cdot  \min_{v \in \mathcal{S}^{n-1}}\{\inner{\theta}{\Xi_v}\}.
   \end{equation*}
With the following lower bound for the numerator in~\eqref{eq:cond_opt}, \[
\max_{v \in \mathcal{S}^{n-1}}\{\inner{\mu}{\Xi_v}\} \geq \max_k\{\mu_k\|A_k\|_2\} = \max_k\{\mu_k\} \ge 1/d,
\]
it thus holds that 
\[
 \kappa^* \ge \frac{1}{d^2 \cdot \min_{v \in \mathcal{S}^{n-1}}\{\inner{\theta}{\Xi_v}\} \}} \ge d^2 \kappa(A(\theta)),
\]
where we used $\|A(\theta)\|_2 \le 1$.
%   
%    
%    Therefore, for the particular $\mu^*$ that optimizes \eqref{eq:cond_opt},
%    \begin{equation}\label{eq:kappa_A_theta_bound}
%    \kappa(A(\theta)) \leq \frac{1}{\min_{v \in \mathcal{S}^{n-1}}\{\inner{\theta}{\Xi_v}\}}  \leq \frac{d}{\min_{v \in \mathcal{S}^{n-1}}\{\inner{\mu^*}{\Xi_v}\}}.\end{equation}
%    We need to relate the numerator of \eqref{eq:kappa_A_theta_bound}, which is $d$, to the numerator of \eqref{eq:cond_opt}. Observe that for any $\mu \in \Delta^d$,
%    \[\max_{v \in \mathcal{S}^{n-1}}\{\inner{\mu}{\Xi_v}\} \geq \max_k\{\mu_k\|A_k\|_2\},\]
%    which yields $d \max_{v \in \mathcal{S}^{n-1}}\{\inner{\mu}{\Xi_v}\} \geq 1$.
%     
% Therefore, for the particular $\mu^*$ that optimizes \eqref{eq:cond_opt}, combining the above bound with \eqref{eq:kappa_A_theta_bound} results in
%    \[\kappa(A(\theta)) \leq \frac{ d^2 \max_{v \in \mathcal{S}^{n-1}}\{\inner{\mu^*}{\Xi_v}\}}{\min_{v \in \mathcal{S}^{n-1}}\{\inner{\mu^*}{\Xi_v}\}} \leq d^2 \kappa^*.\]
   Since $\kappa(\tilde A( \theta)) \leq \kappa(A(\theta)) + \Ocal(\epsilon)$, 
\[\kappa(\tilde X) = \kappa(\tilde L) = \kappa(\tilde A(\theta))^{1/2}\leq  \kappa(A(\theta))^{1/2} + \Ocal(\epsilon) \leq d \sqrt{\kappa^*} + \Ocal(\epsilon).\]
\end{proof}

%% file: implementation_details_RFFDIAG.tex
\section{Implementation details}
\label{sec:implementation_details}

In this section, we discuss important implementation details and improvements for Algorithm~\ref{alg:RSDC}.

\subsection{RSDC details}

%In this subsection, we discuss how we solve Line 3 in Algorithm \ref{alg:RSDC} through a generalized eigenvalue problem and how we can boost the success probability by multiple trials. The implementation details of Algorithm \ref{alg:RSDC} are summarized in Algorithm \ref{alg:detailed_RSDC}.

Line 3 of Algorithm \ref{alg:RSDC} requires the simultaneous diagonalization by congruence of two random linear combinations $\big(\tilde A(\theta), \tilde A(\mu)\big)$ for a nearly SDC family $(\tilde A_1, \ldots, \tilde A_d)$. As seen for the matrix pair~\eqref{example:nonsdc}, this might not be possible even for arbitrarily small perturbations of SDC families. Thus, one needs to assume that $\big(\tilde A(\theta), \tilde A(\mu)\big)$ is SDC. If, additionally, this pair is regular, there is a strong link between the transformation matrix and the matrix of eigenvectors.

\begin{lemma}\label{prop:diag_iff_sdc} Let $(A,B)$ be a
    symmetric regular SDC pair. Then the following hold:
    \begin{enumerate}
        \item [(i)] If $X$ is an invertible matrix such that $X^TAX,X^TBX$ are diagonal then $X$ is a matrix of eigenvectors of $(A,B)$, that is, there are diagonal matrices $D_A$, $D_B$ such that $AX D_B = B X D_A$.
        \item[(ii)] If $X$ is an invertible matrix of eigenvectors of $(A,B)$ then there exists an orthogonal matrix $Q$ such that $(XQ)^TA (XQ),(XQ)^TB (XQ)$ are diagonal.
    \end{enumerate}   
\end{lemma}

\begin{proof}
(i) Defining $D_A := X^TAX$, $D_B := X^TBX = D_B$, we  immediately have that
$AXD_B = X^{-T} D_A D_B = X^{-T} D_B D_A  = BX D_A$.

(ii) Considering the relation $AX D_B = B X D_A$ for an eigenvector matrix $X$, the regularity assumption implies that there is a linear combination $D(\mu):=\mu_A D_A + \mu_B D_B$ with $\mu_A \neq 0$ and $\mu_B \neq 0$ such that $D(\mu)$ is invertible. Then $X^T ( \mu_A  A + \mu_B B) X D_A D(\mu)^{-1}  = X^T A X $ implies that
\begin{align*}
      X^T A X  \cdot  X^T ( \mu_A  A + \mu_B B) X &= X^T ( \mu_A  A + \mu_B B) X D_A D(\mu)^{-1} X^T ( \mu_A  A + \mu_B B) X \\
      &= X^T ( \mu_A  A + \mu_B B) X D(\mu)^{-1} D_A X^T ( \mu_A  A + \mu_B B) X \\
      &= X^T ( \mu_A  A + \mu_B B) X \cdot X^T A X,
   \end{align*}
where we used the symmetry of the involved factors in the last equality. Therefore, $X^TAX$ and $X^TBX$ commute, implying that there exists an orthogonal matrix $Q$ such that $Q^TX^TAXQ$ and $Q^TX^TBXQ$ are diagonal. 
\end{proof}
Lemma~\ref{prop:diag_iff_sdc} implies for a regular symmetric pair that being SDC is equivalent to being \emph{diagonalizable} \cite[p.297]{StewartSun1990}.
If all eigenvalues are simple then $X$ is uniquely determined up to column reordering and scaling. See also \cite[Theorem 2]{TensorDforSP} for a more detailed discussion on the conditions under which the generalized eigenvalue problem approach yields a unique $X$. This implies that the matrix $Q$ in Lemma~\ref{prop:diag_iff_sdc} (ii) can be chosen to be the identity matrix. In other words, solving the generalized eigenvalue problem
\begin{equation} \label{eq:gep}
 \tilde A(\theta) \tilde X \Theta =  \tilde A(\mu) \tilde X \Lambda
\end{equation}
directly gives the matrix $\tilde X$ that diagonalizes $(\tilde A(\mu), \tilde A(\theta))$ by congruence. In general, this does not hold and $Q$ can be computed by jointly diagonalizing the commuting symmetric matrices $\tilde X^T \tilde A(\mu) \tilde X$ and $\tilde X^T \tilde A(\theta) \tilde X$, using the algorithms from~\cite{hekressner2024randomized,Sutton23}. For none of the experiments reported in Section~\ref{sec:numerical_experiments}, this orthogonal joint diagonalization step was necessary. Additionally, generlized eigenvalues are generically distinct. Thus, we only use the generalized eigenvalue solver {\tt dggev} by LAPACK~\cite{LAPACK}, which is based on the QZ algorithm~\cite{kressner05,matrixcomputation}, to obtain $\tilde X$ from~\eqref{eq:gep} and keep the orthogonal joint diagonalization step optional. As all the matrices are real, all possible complex eigenvalues and eigenvectors of \eqref{eq:gep}  will appear in complex conjugated pairs. For a
complex eigenvalue, {\tt dggev} returns the real and imaginary parts of the corresponding eigenvector.

The success probability of Algorithm \ref{alg:RSDC} can be easily boosted by considering several independent random linear combinations and choosing the candidate that minimizes the off-diagonal error~\eqref{eq:off_diagonal_loss}. As~\eqref{eq:off_diagonal_loss} is not invariant under column scaling,
it is necessary to normalize the columns of the returned transformation matrix before comparing the quality of different trials.

\begin{algorithm}[ht]
    \caption{\textbf{R}andomized \textbf{S}imultaneous \textbf{D}iagonalization via \textbf{C}ongruence (RSDC)}
    \textbf{Input:} \text{Nearly SDC family $(\tilde{A}_1,\ldots,\tilde{A}_d)$, number of trials $L$.}\\
     \textbf{Output:} \text{Invertible matrix $\tilde{X}$ such that $\tilde{X}^T \tilde A_k \tilde{X}$ is nearly diagonal for $k = 1,\ldots,d$.} \\[-0.5cm]
    \begin{algorithmic}[1]
    \label{alg:detailed_RSDC}
    \FOR{$i=1$ to $L$}
        \STATE Draw $\mu, \theta \in \R^d$ from specific distribution.
        \STATE Compute $\tilde{A}(\mu) = \mu_1 \tilde A_1 + \cdots + \mu_d \tilde A_d$, $\tilde{A}(\theta) = \theta_1 \tilde A_1 + \cdots + \theta_d \tilde A_d$.
        \STATE Compute matrix of eigenvectors $\tilde X_i$ for $(\tilde A(\mu), \tilde A(\theta))$ via QZ algorithm.
        \STATE [optional] Compute orthogonal matrix $\tilde Q$ that jointly diagonalizes $\tilde X_i^T \tilde A(\mu) \tilde X_i$, $\tilde X_i^T \tilde A(\theta) \tilde X_i$ and set $\tilde X_i \leftarrow \tilde X_i \tilde Q$.
        \STATE Normalize $\tilde X_i$ so that each column has norm $1$.
    \ENDFOR
    \STATE Let $i^* = \argmin_{i}\Big\{ \sum\limits_{k = 1}^{d}\fnormbig{\offdiag(\tilde{X}_i^T \tilde{A}_k \tilde{X}_i)}^2\Big\}$.
    \RETURN $\tilde X = \tilde{X}_{i^*}$
    \end{algorithmic}
    \end{algorithm}
The considerations of this section lead to Algorithm~\ref{alg:detailed_RSDC}, which can be viewed as multiple iterations of Algorithm \ref{alg:RSDC}.
Note that in Line $2$ of Algorithm \ref{alg:detailed_RSDC}, both $\mu,\theta$ are Gaussian random vectors if we only know that the input family is regular, while we choose $\theta = [1/d,\ldots,1/d]$ if the input family is known to be positive definite.
\subsection{Iterative refinement with FFDIAG}
\label{sec:FFDIAG}

When the noise level is high and the dimension $n$ is large, Algorithm~\ref{alg:detailed_RSDC} alone -- especially with a limited number of trials -- may not provide sufficient accuracy, as indicated by the $\mathcal{O}(n^2\epsilon)$ factor in Theorem~\ref{thm:pd_prob_bound} and also demonstrated by numerical experiments in Section~\ref{sec:numerical_experiments}. To improve the accuracy of Algorithm~\ref{alg:detailed_RSDC} more efficiently instead of merely increasing the number of trials, we propose to combine it with optimization-based algorithms.   Our results from Section~\ref{sec:robust_recovery} provide some indication that RSDC can be expected to deliver decent initial guesses for an optimization method; at least for moderate $n$ and small $\epsilon$ when compared to the trivial initial guess. As we will see in the numerical experiments in Section~\ref{sec:numerical_experiments}, this combined approach can efficiently produce accurate results even for a large dimension $n$ and a high noise level. However, developing a theoretical understanding of how using the output from RSDC as an initial guess improves the recovery result of optimization-based algorithms for SDC remains an open challenge.% Our results from Section~\ref{sec:robust_recovery} provide theoretical justification for this choice. In the orthogonal JD case, we used deflation~\cite{hekressner2024randomized} to improve the output of the randomized algorithm.  Due to the lack of orthogonality of $\tilde X$ in SDC, the benefits of deflation are less clear for Algorithm~\ref{alg:detailed_RSDC}. Instead, we use the output of Algorithm~\ref{alg:RSDC} as a (good) starting point for an optimization algorithm.

The particular optimization algorithm considered in this work is the quasi-Newton method FFDIAG from~\cite{Ziehe03}, which aims at minimizing the off-diagonal error~\eqref{eq:off_diagonal_loss} by multiplicative updates of the form $\tilde X_{i+1} = (I+W)\tilde X_i$ for some carefully chosen $W$;
%Note that FFDIAG does not require the input to be PD.
see~\cite[Algorithm 1]{Ziehe03} for more details. FFDIAG uses the identity matrix as the starting point and, quite remarkably, the need for investigating a smarter initialization is explicitly mentioned in~\cite{Ziehe03}; we believe that RSDC is an excellent candidate. For the implementation of FFDIAG, we follow the library PYBSS\interfootnotelinepenalty=10000\footnote{The library is owned and maintained by Ameya Akkalkotkar and Kevin Brown, available at \url{https://github.com/archimonde1308/pybss}} and further improve the efficiency by vectorizing for loops. The stopping criterion of FFDIAG considered throughout this paper is $\|\tilde X_{i+1} - \tilde X_i\|_F \leq \num{1e-8} $, which is the default used in PYBSS. 

RFFDIAG, the described combination of RSDC with FFDIAG, is summarized in Algorithm \ref{alg:RFFDIAG}. To demonstrate how the randomized initial guess helps FFDIAG, we consider $d=10$ randomly generated SDC matrices of size $100 \times 100$; the matrices are generated in the same way as the noiseless SDC matrices in Section \ref{sec:synthetic_data}.
FFDIAG initialized with the identity matrix requires $47$ iterations to converge. In contrast, when initialized with the output of RDSC, it only requires $1$ iteration.

\begin{algorithm}[H]
\caption{\textbf{R}andomized \textbf{FFDIAG} (RFFDIAG)}
\label{alg:RFFDIAG}
\textbf{Input:} \text{Nearly SDC family $\tilde{\Acal} = (\tilde{A}_1,\ldots,\tilde{A}_d)$, maximum number of iterations $N$.}\\
 \textbf{Output:} \text{Invertible matrix $\tilde{X}$ such that $\tilde{X}^T \tilde A_k \tilde{X}$ is nearly diagonal for $k = 1,\ldots,d$.} \\[-0.5cm]
\begin{algorithmic}[1]
\STATE $\tilde X_0 = \text{RSDC}(\tilde {\mathcal{A}}, 1)$ \emph{\% Calling Algorithm~\ref{alg:detailed_RSDC} with $L=1$.}
\STATE $\tilde X = \text{FFDIAG}(\tilde X_0,  \tilde{\Acal}, N)$ 
\emph{\% Calling FFDIAG with starting point $\tilde X_0$.}
\RETURN $\tilde X$

\end{algorithmic}
\end{algorithm}

% \begin{table}[!hbt!]
%     \begin{center}
%     \caption{Number of iterations comparison FFDIAG vs RFFDIAG}
%     \label{table:rffdiag_iters}
%     \begin{tabular}{||c | c | |}
%     \hline
%     Algorithm & \# of iterations\\
%     \hline
%     FFDIAG & $\num{47}$\\
%     \hline
%     RFFDIAG & $\num{1}$\\
%     \hline
%     \end{tabular}
%     \end{center}
% \end{table}

%% file: numerical_experiments.tex
\section{Numerical experiments} \label{sec:numerical_experiments}

We have implemented the algorithms described in this paper in Python 3.8; the code is available at \url{https://github.com/haoze12345/rsdc}.
Throughout this section, we use $L=3$ trials to boost the success probability of RSDC when using it as a stand-alone algorithm. The number of maximum iterations of RFFDIAG is set to $10$ because it requires only few iterations to converge, as demonstrated in Section \ref{sec:FFDIAG}. We have found these settings to offer a good compromise between accuracy and efficiency. All experiments were carried out on a Dell XPS 13 2-In-1 with an Intel Core i7-1165G7 CPU and 16GB of RAM. All execution times are reported in milliseconds.

In the following, we demonstrate the performance of RSDC and RFFDIAG for synthetic data, image separation, and electroencephalographic recordings. The numerical experiments are organized to be closer and closer to real-world scenarios. Before delving into these extensive numerical experiments, we provide an overview of the alternative algorithms and their implementations.

For FFDIAG, we use our own optimized implementation, as introduced in Section \ref{sec:FFDIAG}. PHAM~\cite{Pham2001} minimizes the loss~\eqref{eq:pham_loss} by decomposing the diagonalizer into $n(n-1)/2$ invertible elementary transformations and minimizing \eqref{alg:RFFDIAG} successively for each elementary transform. PHAM's implementation is available in~\cite{pyriemann}. Note that PHAM is the method of choice for the Blind Source Separation (BSS) task in~\cite{BARTHELEMY2017371}, which we compare to in Section \ref{sec:eye_blinking}. QNDIAG~\cite{Ablin19}, a quasi-Newton's method,  also minimizes loss \eqref{eq:pham_loss} with an efficient approximation of the Hessian. We use its original implementation in \cite{Ablin19}\interfootnotelinepenalty=10000\footnote{Available at \url{https://github.com/pierreablin/qndiag}}. FFDIAG, PHAM and QNDIAG are compared with our novel randomized algorithms on synthetic data in Section \ref{sec:synthetic_data}. UWEDGE~\cite{uwedge} serves as the preferred SDC solver for image separation tasks in \cite{Pfister2019}; we compare to the original implementation from~\cite{Pfister2019} for the quality of image separation and efficiency in Section \ref{sec:imageseparation}. UWEDGE minimizes the loss \eqref{eq:off_diagonal_loss} iteratively as follows: Given the current estimated diagonalizer $\tilde X_i$, compute the best ``mixing" matrix $\tilde V_i$ such that $\tilde V_i\diag(\tilde X_i^TA_k\tilde X_i)\tilde V_i^T$ is as close to $\tilde X_i^TA_k \tilde X_i, i = 1,\ldots,d$ as possible and set $\tilde X_{i+1} = \tilde X_i \tilde V_i^{-T}$.  For all the alternative algorithms, we use the identity as the initial values and keep all parameters, such as stopping criteria, default for all numerical experiments.

\subsection{Synthetic data}
\label{sec:synthetic_data}
In this section, we compare our algorithms with FFDIAG, PHAM and QNDIAG on synthetic data. 

For this experiment, synthetic nearly SDC families $\tilde \Acal = \{\tilde A_k = VD_kV^T + \epsilon E_k \in \reals^{n \times n}\}_{k=1}^{d}$ have been generated as follows. The matrix $V$ is fairly well-conditioned, obtained from normalizing the columns of a Gaussian random matrix. Each diagonal entry of $D_k$ is the absolute value of an i.i.d. standard normal random variable, shifted by $0.01$ to ensure sufficiently strong positivity. We consider three different sizes $(d,n) = (10,10), (100,10), (10,100)$ and three different noise levels $\epsilon_1 = 0$, $\epsilon_2 = 10 ^{-6}$, and $\epsilon_3 = 10^{-3}$. The perturbation directions $E_k$ are Gaussian random matrices normalized such that $\sum_{k=1}^{d}\fnorm{E_k}^2 = 1$.  % The setting $(10,10)$ acts as a sanity check, while $(100,10)$ and $(10,100)$ settings imitate the situation where we get a large number of SDC matrices and a family of SDC matrices with large sizes respectively.
Because the positive definiteness of each matrix is assumed by QNDIAG and PHAM, we enforce PD by repeatedly generating $E_k$ until all matrices $\tilde A_k$ are positive definite. We always scale the columns of the output $\tilde X$ to have norm $1$ so that the output error is comparable among different algorithms. The obtained results are shown in the Table \ref{table:RSDC_10_10}--\ref{table:RSDC_10_100}. The execution times and errors for each setting are averaged over 100 runs with the same family of matrices. We have repeated the same experiments for several randomly generated nearly SDC families (with the same settings) to verify that the results shown in Table \ref{table:RSDC_10_10}--\ref{table:RSDC_10_100} are representative. 

\input{table_error_runtime_rffdiag}

We have also tested the algorithms on relatively ill-conditioned matrices. For this purpose, we set $\epsilon = 0$, $d = 30$, and $n = 20$. The matrix $V$ is generated as described above, while the diagonal entries of $D_k$ are a random permutation of the vector \[\begin{bmatrix} 10^{0} & 10^{8/(n-1)} &  10^{16/(n-1)} & \ldots &  10^{8}\end{bmatrix}^T.\]  The results are shown in Table \ref{table:ill_conditioned_rsdc}. 
\input{table_error_rumtime_rffdiag_ill_conditioned}

From all the above experiments on synthetic data, we can conclude that RFFDIAG is significantly more efficient than PHAM and QNDIAG, while reaching a level of accuracy that is comparable or better. PHAM and QNDIAG struggle to obtain good accuracy for ill-conditioned matrices, while they pose no problem for RFFDIAG. Among the algorithms considered in this paper, RFFDIAG offers the best compromise between accuracy and efficiency.

The numerical results reported in Tables \ref{table:RSDC_10_10} -- \ref{table:RSDC_10_100} show that the error of RSDC increases with $n$. This growth with $n$ is also reflected by the constants in our theoretical results (Theorem \ref{thm:npd_prob_bound} and \ref{thm:pd_prob_bound}), as well as the sensitivity analysis performed in~\cite[Eqn. (41)]{Afsari08} and \cite[Theorem 1.4]{Pencilbased}. Table \ref{table:RSDC_10_100} indicates that this puts RSDC at a disadvantage when the input error and the dimension $n$ become larger and this makes it necessary to refine the result of RSDC by a subsequent optimization, as done by RFFDIAG.

\subsection{Application: Blind Source Separation}

In Blind Source Separation (BSS), the observed $n$ signals $x_i(t)$, $i=1,\ldots,n$, are assumed to be a linear mixture of $n$ source signals
$s_j(t)$:
\[x_i(t) = \sum_{j = 1}^{n} A_{ij}s_j(t), \text{or } x(t) = As(t)\]
with some (unknown) non-singular matrix $A \in \reals^{n \times n}$.  The source signals $s_1(t), \ldots, s_n(t)$ are assumed to be jointly stationary random processes, that there is at most one Gaussian source, and that for each $t$, the signals $s_1(t),\ldots, s_n(t)$ are mutually independent random variables. The task of BSS is to estimate a matrix $B\in \R^{n\times n}$ only from the observed signals such that the unmixed signals $(Bx)_j(t)$ is a scalar multiple of some source signal $s_j(t)$. Under these assumptions, it is easy to see that the covariance matrix $\mathbb{E}[xx^T] = A\mathbb{E}[ss^T]A^T $ is diagonalized by congruence via $A^{-T}$ as $s_1(t),\ldots, s_n(t)$ are mutually independent.  Other second-order statistics, e.g., time-lagged covariance matrices \cite{Belouchrani1997}, Fourier cospectra~\cite{BARTHELEMY2017371}, or even covariance matrices of different signal segments~\cite{Pham2001IEEE},  of the observed signals $x_i(t)$ share the same simultaneous diagonalizer via congruence as the covariance matrix. Therefore, the unmixing matrix $B$ can be estimated by performing SDC on those second-order statistics. See, e.g., \cite{BSSsurvey} for an overview of BSS and other SDC families given the observed signals. The experiments conducted in the following two subsections are two real-world examples of BSS.

\subsection{Real data: image separation} \label{sec:imageseparation}

We perform an image separation experiment, following Example 1 from~\cite{Pfister2019}. The source signals in this example are vectorized images of landscapes. We mixed the same $4$ photos of landscapes as in \cite{Pfister2019} with a random Gaussian mixing matrix $A \in \reals^{4 \times 4}$.

We then compute the covariance matrices of different signal segments of the observed mixed signals, which results in $1350$ matrices of size $4 \times 4$  to be jointly diagonalized. We apply RSDC and RFFDIAG to this family of nearly SDC matrices and compare them with UWEDGE, which is the solver used in~\cite{Pfister2019}. The obtained results are shown in Figure \ref{fig:image_separartion}. The execution time, averaged over 100 runs for the same mixing matrix $A$, of each algorithm is reported in Table~\ref{table:image_separartion_runningtime}. In this experiment, only the execution times of the SDC part are measured. 
Visually, the separation achieved by RSDC is poor. In contrast, both UWEDGE and RFFDIAG lead to almost perfect reconstruction of the original images. Our new algorithm RFFDIAG is nearly two times faster than UWEDGE.
\begin{figure}[ht]
    \centering
    \includegraphics[width = \linewidth]{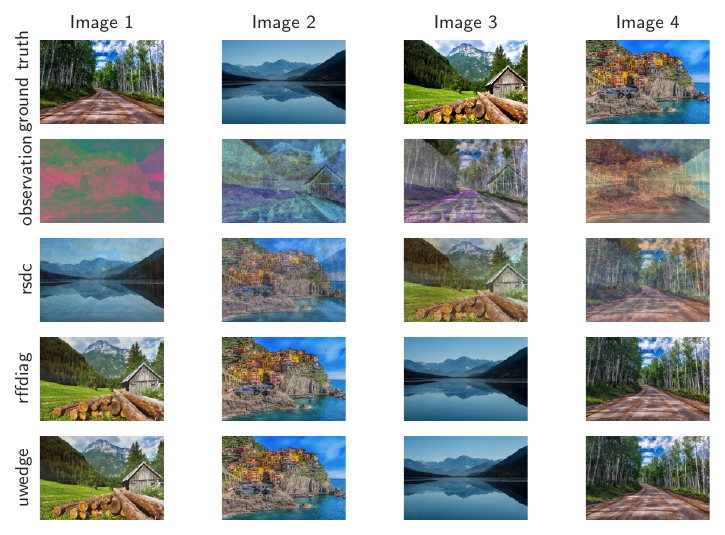}
    \caption{Results for image separation example from Section~\ref{sec:imageseparation}. The first row shows the original images, the second row the mixed images, and the last 3 rows show the unmixed images obtained by RSDC, RFFDIAG, and UWEDGE.}
    \label{fig:image_separartion}
\end{figure}
\input{table_image_separation_runningtime_rffdiag}

\subsection{Real data: electroencephalographic recordings}
\label{sec:eye_blinking}

Finally, we test our randomized SDC algorithms on real electroencephalographic (EEG) recordings. In Human EEG, the signals observed at the different electrodes on the scalp are approximately linear mixtures of the source signals in the brain~\cite[Chapter 8.2]{Congedo2014}. Specifically, in~\cite{BARTHELEMY2017371} it is suggested that the eye-blinking noise signal can be separated by performing SDC on the Fourier cospectra of the EEG recording. This leads to a family of $d=33$ matrices of size $17\times 17$. We follow the implementation in~\cite{pyriemann},
%\footnote{\href{https://pyriemann.readthedocs.io/}{https://pyriemann.readthedocs.io/}},
where the involved SDC solver is PHAM. The power spectrum and the topographic map of the estimated blink source signal with this default choice are shown in Figure \ref{fig:pham_blink_source}. We can see that, from the topographic map of the source signal on the scalp, the source signal is indeed concentrated among the eyes, which confirms that this signal corresponds to eye-blink. 

We applied RSDC and RFFDIAG to the same data and the results are shown in Figures~\ref{fig:rsdc_blink_source} and~\ref{fig:rffdiag_blink_source} respectively. From the figures, we can see that both the algorithms identify the blink signals successfully and their power spectra are also close to the one estimated by the default solver. However, we can observe from the topographic map that the result of RSDC is more noisy.
\begin{figure}
    \centering
    \includegraphics[width = \linewidth]{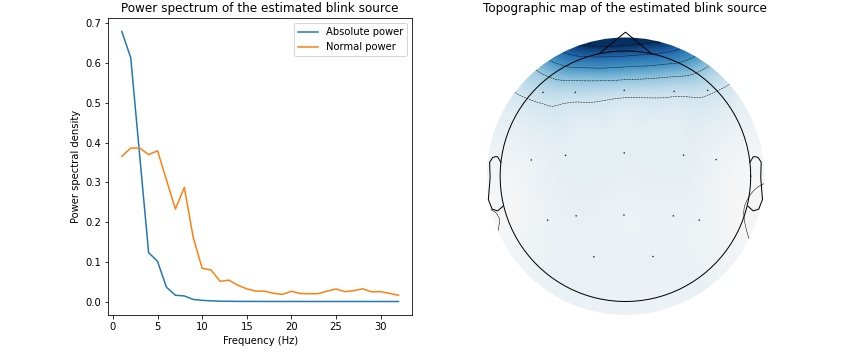}
    \caption{Power spectrum and topographic map of the blink source by PHAM}
    \label{fig:pham_blink_source}
\end{figure}

\begin{figure}
    \centering
    \includegraphics[width = \linewidth]{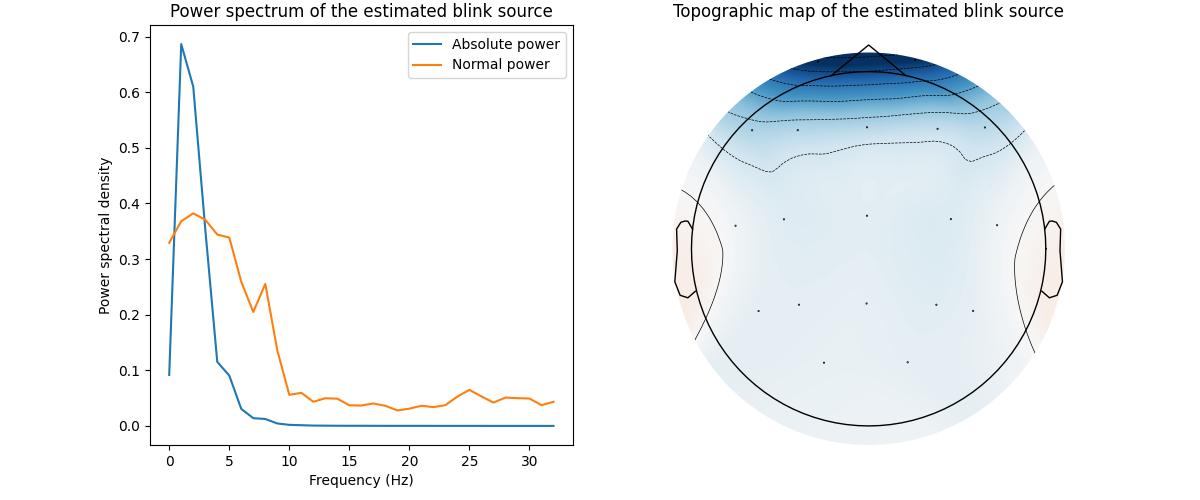}
    \caption{Power spectrum and topographic map of the blink source by RSDC}
    \label{fig:rsdc_blink_source}
\end{figure}

\begin{figure}
    \centering
    \includegraphics[width = \linewidth]{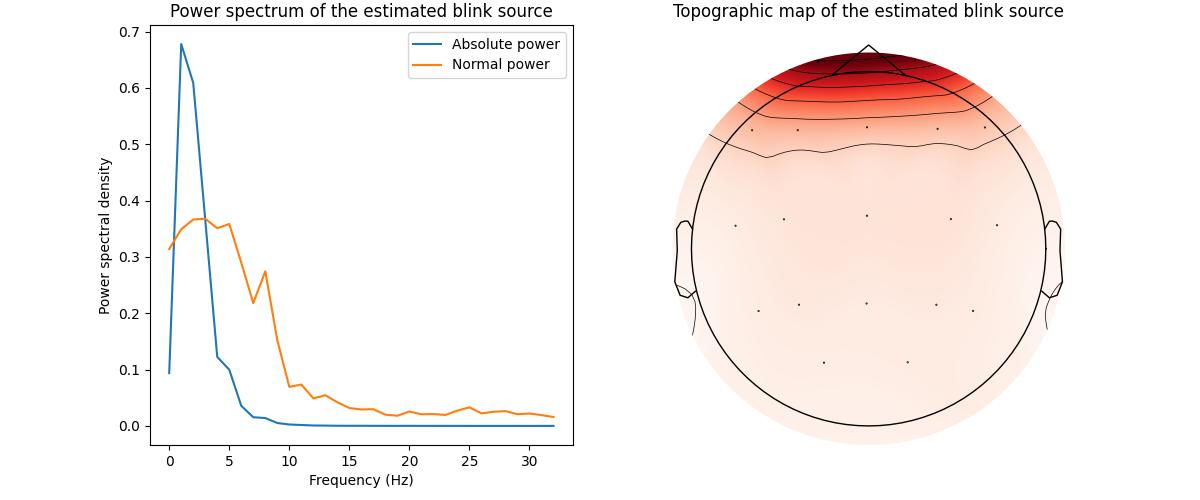}
    \caption{Power spectrum and topographic map of the blink source by RFFDIAG}
    \label{fig:rffdiag_blink_source}
\end{figure}

Next, we compare the execution time of different SDC algorithms involved in this BSS task, averaged over $100$ repeated runs; see Table \ref{table:eeg_runningtime}.  In particular, RFFDIAG yields significantly faster performance than PHAM, while achieving a comparable level of accuracy.
\input{table_eeg_runningtime_FFDIAG}

%% file: table_error_runtime_rffdiag.tex
\begin{table}[!hbt!]
\begin{center}
\caption{Comparison for synthetic data ($d=10, n=10$)}
\label{table:RSDC_10_10}
\begin{tabular}{||c|S[table-format=2.2]|c|S[table-format=2.3]|c|S[table-format=2.2]|c||}
    \hline
    Name & {Time $\epsilon_1$} & Error $\epsilon_1$ & {Time $\epsilon_2$} & Error $\epsilon_2$ &{Time $\epsilon_3$} &Error $\epsilon_3$\\   
    \hline
FFDIAG & 1.53 & $\num{4.14e-11}$ & 1.88 & $\num{9.11e-07}$ & 1.27 & $\num{9.26e-4}$\\

PHAM & 27.75 & $\num{2.38e-09}$ & 26.06 & $\num{1.16e-06}$ & 26.97 & $\num{1.13e-3}$\\

QNDIAG & 1.30 & $\num{3.32e-09}$ & 1.30 & $\num{1.16e-06}$ & 1.10 & $\num{1.13e-3}$\\

RSDC & 0.42 & $\num{7.06e-15}$ & 0.32 & $\num{4.41e-06}$ & 0.50 & $\num{4.57e-3}$\\

RFFDIAG & 0.24 & $\num{3.42e-16}$ & 0.45 & $\num{9.11e-07}$ & 0.48 & $\num{9.26e-4}$\\
    \hline
    \end{tabular}
\end{center}
\end{table}

\begin{table}[!hbt!]
\begin{center}
\caption{Comparison for synthetic data ($d=100, n=10$)}
\label{table:RSDC_100_10}
\begin{tabular}{||c|c|c|c|c|c|c||}
    \hline
    Name & Time $\epsilon_1$ & Error $\epsilon_1$ & Time $\epsilon_2$ & Error $\epsilon_2$ &Time $\epsilon_3$ &Error $\epsilon_3$\\   
    \hline
FFDIAG & 5.05 & $\num{7.79e-11}$ & 4.70 & $\num{1.14e-06}$ & 5.84 & $\num{1.14e-3}$\\

PHAM & 42.28 & $\num{1.42e-13}$ & 33.97 & $\num{1.18e-06}$ & 40.17 & $\num{1.15e-3}$\\

QNDIAG & 3.61 & $\num{1.98e-08}$ & 4.02 & $\num{1.18e-06}$ & 3.74 & $\num{1.15e-3}$\\

RSDC & 0.71 & $\num{2.31e-14}$ & 0.63 & $\num{5.05e-06}$ & 0.54 & $\num{5.29e-3}$\\

RFFDIAG & 0.70 & $\num{1.56e-15}$ & 0.63 & $\num{1.14e-06}$ & 1.41 & $\num{1.14e-3}$\\
    
    \hline
    \end{tabular}
\end{center}
\end{table}

\begin{table}[!hbt!]
\begin{center}
\caption{Comparison for synthetic data ($d=10, n=100$)}
\label{table:RSDC_10_100}
\begin{tabular}{||c|S[table-format=4.2]|c|S[table-format=4.2]|c|S[table-format=4.2]|c||}
    \hline
    Name & {Time $\epsilon_1$} & Error $\epsilon_1$ & {Time $\epsilon_2$} & Error $\epsilon_2$ &{Time $\epsilon_3$} &Error $\epsilon_3$\\
    \hline
FFDIAG & 120.48 & $\num{4.99e-11}$ & 126.93 & $\num{1.08e-06}$ & 146.54 & $\num{9.61e-4}$\\

PHAM & 6930.17 & $\num{1.08e-07}$ & 6967.46 & $\num{1.37e-06}$ & 8161.35 & $\num{1.19e-3}$\\

QNDIAG & 123.05 & $\num{2.21e-13}$ & 586.69 & $\num{1.36e-06}$ & 470.17 & $\num{1.20e-3}$\\

RSDC & 47.07 & $\num{1.27e-13}$ & 54.39 & $\num{5.38e-05}$ & 62.89 & $\num{1.08e-2}$\\

RFFDIAG & 23.44 & $\num{1.14e-15}$ & 31.42 & $\num{1.08e-06}$ & 64.22 & $\num{9.94e-4}$\\
    
    \hline
    \end{tabular}
\end{center}
\end{table}

%% file: table_error_rumtime_rffdiag_ill_conditioned.tex
\begin{table}[!hbt!]
\begin{center}
\caption{Comparison for synthetic, ill-conditoned matrices}
\label{table:ill_conditioned_rsdc}
\begin{tabular}{||c|S[table-format=2.1]|c||}
    \hline
    Name & {Time} & {Error} \\
    \hline
    FFDIAG & 4.58 & $\num{9.24e-12}$\\

    PHAM & 799.21 & $\num{5.66e-1}$\\
    
    QNDIAG & 38.77 & $\num{1.74e-11}$\\
    
    RSDC & 2.85 & $\num{3.44e-14}$\\
    
    RFFDIAG & 0.99 & $\num{1.03e-15}$\\
    
    \hline
\end{tabular}
\end{center}
\end{table}

%% file: table_image_separation_runningtime_rffdiag.tex
\begin{table}[!hbt!]
\begin{center}
\caption{Average running time in image separation}
\label{table:image_separartion_runningtime}
\begin{tabular}{||c | S[table-format = 2.2] | |}
\hline
Algorithm name & {Avg running time(ms)}\\
\hline
RSDC & 2.32\\
\hline
RFFDIAG & 23.39\\
\hline
UWEDGE & 46.14\\
\hline

\end{tabular}
\end{center}
\end{table}

%% file: table_eeg_runningtime_FFDIAG.tex
\begin{table}[!hbt!]
\begin{center}
\caption{Average running time in eye-blind signal separation}
\label{table:eeg_runningtime}
\begin{tabular}{||c | S[table-format = 3.2] | |}
\hline
Algorithm name & {Avg running time(ms)}\\
\hline
PHAM & 275.49\\
\hline
RSDC & 1.42\\
\hline
RFFDIAG & 2.69\\
\hline
\end{tabular}
\end{center}
\end{table}

%% file: conclusion.tex
\section{Conclusions}

In this paper, we have proposed and analyzed RSDC, a novel randomized algorithm for performing (approximate) simultaneous diagonalization by congruence. Our numerical experiments show that this algorithm is best used in combination with an optimization method that uses the output of RSDC as starting point. The resulting algorithm, RFFDIAG, appears to offer a good compromise between efficiency and accuracy, outperforming existing solvers; sometimes by a large margin. While empirical results are promising, a rigorous theoretical analysis of RFFDIAG's robustness remains an important direction for future research.
%This superior performance We demonstrate that RSDC can SDC an exactly SDC family with probability $1$ and with high probability SDC a nearly SDC family with an output error on the same level of the input error. Through an extensive series of numerical experiments  on synthetic and real data, we show that RSDC is the method to go to when the input error is of small magnitude, while RFFDIAG reaches the state-of-the art performance in terms of both accuracy and efficiency across all examples.

\begin{paragraph}{Acknowledgments.}
The authors thank the referees and the editor for helpful remarks, which
improved the presentation of this manuscript. The authors also thank Nela Bosner, University of Zagreb, for inspiring discussions related to this work.
\end{paragraph}